\documentclass[11pt]{article}

\usepackage[utf8]{inputenc}
\usepackage[T1]{fontenc}
\usepackage{lmodern}
\usepackage{amsmath,amssymb,amsthm,mathtools}
\usepackage{mathrsfs}
\usepackage{bm}
\usepackage{enumitem}
\usepackage{natbib}
\usepackage{url}
\usepackage[hidelinks]{hyperref}
\usepackage[margin=1in]{geometry}

\numberwithin{equation}{section}

\newtheorem{theorem}{Theorem}[section]
\newtheorem{proposition}[theorem]{Proposition}
\newtheorem{lemma}[theorem]{Lemma}
\newtheorem{corollary}[theorem]{Corollary}
\newtheorem{definition}[theorem]{Definition}

\newcommand{\R}{\mathbb{R}}
\newcommand{\Pcal}{\mathcal{P}}
\newcommand{\Gcal}{\mathcal{G}}
\newcommand{\bx}{\bm{x}}

\newcommand{\Hess}{\operatorname{Hess}}
\newcommand{\gradS}{\operatorname{grad}_{\mathbb S^2}}
\newcommand{\dd}{\,\mathrm{d}}
\newcommand{\keywords}[1]{\par\medskip\noindent\textbf{Keywords:} #1\par}

\setcitestyle{authoryear,open={[},close={]},semicolon,aysep={,}}

\hypersetup{
  pdftitle={An Elegant Analytical Resolution of the Sprott--Zeraoulia Conjecture for Three-Dimensional Quadratic Differential Systems with Symmetric Jacobian Matrices},
  pdfauthor={Marcelo Messias and Rafael Paulino Silva},
  pdfkeywords={Sprott--Zeraoulia conjecture; symmetric Jacobian matrix; polynomial gradient flow; Lojasiewicz inequality; Poincare compactification}
}

\title{An Elegant Analytical Resolution of the\\
Sprott--Zeraoulia Conjecture for Three-Dimensional  \\
Quadratic Differential Systems
with Symmetric Jacobian Matrices}

\author{Marcelo Messias\thanks{Corresponding author: \texttt{marcelo.messias1@unesp.br}} \and
Rafael Paulino Silva\thanks{UNESP Postdoctoral Program; \texttt{paulino.silva@unesp.br}}\\[0.6em]
\small Department of Mathematics and Computer Science, School of Technology and Sciences,\\
\small S\~ao Paulo State University (UNESP), Rua Roberto Simonsen 305,\\
\small 19060-900 Presidente Prudente, S\~ao Paulo, Brazil}

\date{}

\begin{document}
\maketitle
\begin{abstract}
The Sprott--Zeraoulia conjecture, formulated on the basis of numerical
experiments, states that three-dimensional quadratic continuous-time
systems with symmetric Jacobian matrices cannot exhibit chaotic
behavior.  We give an analytical resolution of this conjecture in terms
of the principal precise notions of chaos used for continuous dynamical
systems.  The starting point is the classical, but decisive, fact that
symmetry of the Jacobian on $\R^3$ is equivalent to a global gradient
representation.  Hence every system in the conjectured class can be
written as $\dot{\bx}~=~\nabla V(\bx)$, where $V$ is a polynomial of degree
at most three, and the conjecture becomes a problem about cubic
polynomial gradient flows.

For bounded trajectories, the monotonicity property of gradient systems 
first confines every omega-limit set to a critical level of $V$.  When
the equilibria on that level are isolated, convergence follows from the
connectedness of the omega-limit set.  In the general case, which may
include curves or surfaces of equilibria, the \L{}ojasiewicz gradient
inequality implies finite length and convergence of every bounded
forward trajectory to a single equilibrium.  Consequently, compact
invariant sets support no nonstationary recurrence, positive-entropy
measure, topological transitivity, Smale horseshoe, or Devaney,
Auslander--Yorke, Li--Yorke, mean Li--Yorke, and distributional chaos.

Solutions escaping to infinity are treated through the Poincar\'e
compactification.  The induced flow on the sphere at infinity is the
spherical gradient flow generated by the highest homogeneous part of
$V$, and therefore cannot support chaotic limit dynamics either.
Finally, we connect the gradient formulation with earlier
Darboux-theoretic criteria: invariant algebraic surfaces satisfy
$\langle\nabla V,\nabla f\rangle=Kf$, and an invariant affine plane with
constant cofactor produces an exact tangential--normal decomposition of
the potential and the flow.
\end{abstract}



\section{Introduction}
\label{sec:introduction}

Three-dimensional quadratic differential systems occupy a central
position in nonlinear dynamics.  Their algebraic degree is low enough
to permit explicit analysis, but their phase portraits may contain
bifurcations, invariant manifolds, strange attractors, and chaotic
dynamics.  Against this background, Zeraoulia and Sprott formulated the
conjecture that a quadratic continuous-time system in $\R^3$ whose
Jacobian matrix is symmetric at every point cannot be chaotic
\cite{ZeraouliaSprott2014}.  The conjecture arose from numerical
experiments in the general $19$-parameter family of such systems.  
For this purpose, they randomly chose the parameter values and three initial conditions from a Gaussian distribution with mean zero and variance one, and computed the largest Lyapunov exponent for approximately $10^9$ (one billion) cases. They found that the vast majority of the solutions were unbounded (one or more variables exceeded 1000), whereas a small subset, of the order of one percent, was attracted to a stable equilibrium. Not a single case exhibited a significantly positive Lyapunov exponent (greater than 0.001). Nevertheless, a conclusive mathematical proof of this quite general conjecture has remained lacking. Note that the
symmetry of the linearization at each equilibrium implies that all
local eigenvalues are real, thereby excluding saddle-focus equilibria
and the classical Shilnikov mechanism.  This local observation,
however, does not by itself rule out other mechanisms of recurrent or
chaotic dynamics.

Several partial positive answers were subsequently obtained by
algebraic methods.  In \cite{MessiasSilva2018}, an invariant algebraic
surface with nonzero constant cofactor was used to confine limit sets
and establish nonchaotic behavior for a significant subclass of the
Sprott--Zeraoulia family.  Related criteria were developed for
polynomial jerk equations and Lorenz-like systems
\cite{MessiasSilva2020,MessiasSilva2022}, and the recent result of
\cite{SilvaMessias2026} allows several invariant algebraic surfaces
whose cofactors satisfy an appropriate linear relation.  These works
show that Darboux polynomials, invariant surfaces, and exponential
factors may provide a geometric skeleton that prevents or constrains
chaotic dynamics.  They do not, however, cover every quadratic system
with symmetric Jacobian.

The present paper starts from a different structural observation. Take the quadratic continuous-time differential systems with symmetric Jacobian matrix denoted by $\dot x=F(x)$. On
the simply connected space $\R^3$, the condition
\begin{equation}
 DF(\bx)=DF(\bx)^T
 \label{eq:intro-symmetric}
\end{equation}
is equivalent to the existence of a globally defined scalar potential
$V$ such that
\begin{equation}
 F=\nabla V.
 \label{eq:intro-gradient}
\end{equation}
For a quadratic polynomial field, the potential has degree at most
three.  Thus the class introduced by Zeraoulia and Sprott is precisely
the class of cubic polynomial gradient flows in $\R^3$.

The implication \eqref{eq:intro-symmetric}--\eqref{eq:intro-gradient}
is classical.  We emphasize that the contribution of this paper is not
the Poincar\'e lemma itself, nor the isolated observation that gradient
flows possess a monotone potential.  Rather, the contribution is to
recognize that this classical structure applies to the entire
conjectured family and to develop all the additional arguments needed
to convert a numerically motivated and informally stated conjecture
into a precise global theorem.  Four issues must be addressed.  First,
the term \emph{chaos} has no single universal definition, so the
conclusion must be formulated and proved separately for the principal
topological and measure-theoretic notions used in the literature.
Second, monotonicity alone only places an omega-limit set inside a
critical level; it does not immediately imply convergence when the
critical set contains continua of equilibria.  Third, bounded-orbit
arguments do not treat solutions that escape to infinity, possibly in
finite original time.  Fourth, a satisfactory resolution should
explain how the global gradient mechanism is related to the invariant
algebraic structures used in the earlier partial results.

The gradient identity
\begin{equation}
 \frac{\dd}{\dd t}V(\bx(t))
 =\langle\nabla V(\bx(t)),\dot{\bx}(t)\rangle
 =\|\nabla V(\bx(t))\|^2
 \geq 0
 \label{eq:intro-monotonicity}
\end{equation}
is the basic mechanism behind the proof.  It immediately excludes
nonconstant periodic orbits, nonstationary recurrent points,
homoclinic trajectories, and heteroclinic cycles.  For a bounded
forward trajectory, it implies that the omega-limit set is contained in
one critical level of $V$.  If the equilibria on that level are
isolated, connectedness of the omega-limit set yields convergence to a
single equilibrium.  The nonisolated case is subtler and is essential
for a complete treatment of polynomial systems, because their
equilibrium sets may contain algebraic curves or surfaces.  We use the
\L{}ojasiewicz gradient inequality to prove that every bounded forward
trajectory has finite length and converges to one equilibrium, without
any isolation or hyperbolicity assumption.

This convergence theorem allows the original conjecture to be replaced
by explicit dynamical statements.  On every compact invariant set, the
recurrent and nonwandering core is stationary; every ergodic invariant
probability measure is a Dirac mass at an equilibrium; and every time
map has zero metric and topological entropy.  We consequently exclude
topological transitivity on nontrivial compact invariant sets, Devaney
and Auslander--Yorke chaos, Li--Yorke and mean Li--Yorke chaos,
distributional chaos, positive-entropy symbolic dynamics, and Smale
horseshoes.  In particular, no compact invariant set can be a
nontrivial transitive strange attractor.  This formulation avoids the
ambiguity of claiming merely that the systems are ``not chaotic'' and
states exactly which standard mechanisms and definitions are ruled
out.

To treat unbounded solutions, we use the Poincar\'e compactification.
Writing the potential as
\[
 V=V_3+V_2+V_1+V_0,
\]
where $V_j$ is homogeneous of degree $j$, the quadratic homogeneous part
of the field is $F_2=\nabla V_3$.  After the standard desingularizing
time change, the vector field induced on the sphere at infinity is
\begin{equation}
 u'=\operatorname{grad}_{\mathbb S^2}
 \bigl(V_3|_{\mathbb S^2}\bigr)(u).
 \label{eq:intro-spherical-gradient}
\end{equation}
Thus the dynamics at infinity is itself an analytic gradient flow.
Its trajectories converge to critical points of the spherical
potential, and its compact invariant sets cannot support recurrence,
positive entropy, horseshoes, or any of the chaos notions considered in
the finite part.  Hence escape to infinity does not provide a hidden
chaotic alternative to the bounded dynamics.

The potential formulation also clarifies and strengthens the
Darboux-theoretic approach.  If $f=0$ is an invariant algebraic surface
with cofactor $K$, then
\begin{equation}
 \langle\nabla V,\nabla f\rangle=Kf.
 \label{eq:intro-darboux}
\end{equation}
On the regular part of $f=0$, the restricted vector field is the
intrinsic gradient of $V|_{\{f=0\}}$.  At a regular equilibrium, the
value of the cofactor is the normal eigenvalue of the Hessian.  In the
important case of an invariant affine plane with constant cofactor, an
orthogonal affine change of coordinates gives the exact decomposition
\begin{equation}
 V(u,v,w)=W(u,v)+\frac{k}{2}w^2,
 \qquad
 (\dot u,\dot v,\dot w)=(W_u,W_v,kw).
 \label{eq:intro-plane-splitting}
\end{equation}
Therefore the invariant plane found by the algebraic criterion is not
an accidental barrier: it is the manifestation of an exact separation
between a planar gradient flow and a linear transverse direction.  The
leading homogeneous part of a Darboux polynomial further induces an
invariant algebraic trace on the sphere at infinity.

The main contributions may therefore be summarized as follows:
\begin{enumerate}[label=\textup{(\roman*)}]
\item the complete reformulation of the Sprott--Zeraoulia class as
      cubic polynomial gradient flows;
\item convergence of every bounded forward trajectory to one
      equilibrium, including nonisolated critical sets through the
      \L{}ojasiewicz gradient inequality;
\item systematic exclusion of the principal compact-invariant notions
      and mechanisms of chaos, together with the classification of
      invariant ergodic measures and the proof of zero entropy;
\item a gradient description of the Poincar\'e sphere at infinity,
      excluding chaotic limit dynamics for escaping solutions; and
\item a structural connection with Darboux theory, including the normal
      interpretation of the cofactor and the exact splitting associated
      with an invariant affine plane.
\end{enumerate}
To the best of our knowledge, these ingredients have not previously
been assembled into a complete analytical resolution of the
Sprott--Zeraoulia conjecture.

The paper is organized as follows.  Section~2 describes the general
quadratic family and the symmetry constraints.  Sections~3 and~4 prove
the global gradient representation and develop the cubic-potential
parametrization.  Section~5 establishes convergence of bounded
trajectories and the finite-space trajectory dichotomy.  Section~6
excludes the standard compact-invariant notions of chaos.  Section~7
studies the Poincar\'e compactification and the dynamics at infinity,
thereby completing the dynamical analysis of bounded and escaping
solutions.  Section~8 then connects the gradient formulation with
invariant algebraic surfaces and the Darboux criteria of
\cite{MessiasSilva2018,MessiasSilva2020,MessiasSilva2022,SilvaMessias2026}.
Section~9 summarizes the conclusions and discusses extensions to
arbitrary polynomial degree and higher dimension.


\section{General Quadratic Polynomial Vector Fields}
\label{sec:general-quadratic}

Throughout the manuscript, a \emph{quadratic polynomial vector field}
means a vector field whose components have degree at most two.  We shall
say that it has \emph{exact degree two} when at least one component has a
nonzero quadratic part.

Consider
\begin{equation}
\dot{\bx}=F(\bx),
\qquad
\bx=(x,y,z)\in\R^3,
\qquad
F=(P,Q,R),
\label{eq:general-system}
\end{equation}
where
\begin{align*}
P(x,y,z)
&=a_0+a_1x+a_2y+a_3z
  +a_4x^2+a_5xy+a_6xz+a_7y^2+a_8yz+a_9z^2,\\
Q(x,y,z)
&=b_0+b_1x+b_2y+b_3z
  +b_4x^2+b_5xy+b_6xz+b_7y^2+b_8yz+b_9z^2,\\
R(x,y,z)
&=c_0+c_1x+c_2y+c_3z
  +c_4x^2+c_5xy+c_6xz+c_7y^2+c_8yz+c_9z^2.
\end{align*}
Thus the vector space of all such systems has dimension $30$ (number of real coefficients in polynomials $P, Q$ and $R$ ).

The Jacobian matrix is
\[
DF(x,y,z)=
\begin{pmatrix}
P_x&P_y&P_z\\
Q_x&Q_y&Q_z\\
R_x&R_y&R_z
\end{pmatrix},
\]
and its entries are affine functions of $(x,y,z)$.

\begin{proposition}\label{prop:constraints}
The Jacobian matrix $DF$ is symmetric for every $(x,y,z)\in\R^3$ if
and only if the following eleven independent relations hold:
\begin{align}
 b_1&=a_2, & b_4&=\frac{a_5}{2}, & b_5&=2a_7,
 & b_6&=a_8, \label{eq:constraints1}\\
 c_1&=a_3, & c_4&=\frac{a_6}{2}, & c_5&=a_8,
 & c_6&=2a_9, \label{eq:constraints2}\\
 c_2&=b_3, & c_7&=\frac{b_8}{2}, & c_8&=2b_9.
 \label{eq:constraints3}
\end{align}
Equivalently, every quadratic vector field with symmetric Jacobian can
be written in the $19$-parameter form
\begin{equation}
\left\{
\begin{aligned}
\dot{x}={}&a_0+a_1x+a_2y+a_3z
+a_4x^2+a_5xy+a_6xz+a_7y^2+a_8yz+a_9z^2,\\
\dot{y}={}&b_0+a_2x+b_2y+b_3z
+\frac{a_5}{2}x^2+2a_7xy+a_8xz
+b_7y^2+b_8yz+b_9z^2,\\
\dot{z}={}&c_0+a_3x+b_3y+c_3z
+\frac{a_6}{2}x^2+a_8xy+2a_9xz
+\frac{b_8}{2}y^2+2b_9yz+c_9z^2.
\end{aligned}
\right.
\label{eq:19parameter-system}
\end{equation}
\end{proposition}

\begin{proof}
The off-diagonal symmetry conditions are
\[
P_y=Q_x,\qquad P_z=R_x,\qquad Q_z=R_y.
\]
Now
\begin{align*}
P_y&=a_2+a_5x+2a_7y+a_8z,
&Q_x&=b_1+2b_4x+b_5y+b_6z,\\
P_z&=a_3+a_6x+a_8y+2a_9z,
&R_x&=c_1+2c_4x+c_5y+c_6z,\\
Q_z&=b_3+b_6x+b_8y+2b_9z,
&R_y&=c_2+c_5x+2c_7y+c_8z.
\end{align*}
Comparison of the coefficients of $1,x,y,z$ gives
\eqref{eq:constraints1}--\eqref{eq:constraints3}.  Notice that the
relation $b_6=c_5$ obtained from the third identity is already implied
by $b_6=a_8=c_5$; hence there are eleven, rather than twelve,
independent constraints.  Substitution of these relations into
\eqref{eq:general-system} yields \eqref{eq:19parameter-system}.
\end{proof}

The free coefficients in \eqref{eq:19parameter-system} are
\[
(a_0,\ldots,a_9),\quad
(b_0,b_2,b_3,b_7,b_8,b_9),\quad
(c_0,c_3,c_9),
\]
which gives $10+6+3=19$ parameters.  Therefore the symmetry condition
has codimension $30-19=11$ in the vector space of all quadratic
polynomial vector fields.  This agrees with the parametrization used in
the original formulation of the conjecture \cite{ZeraouliaSprott2014}.


\section{Symmetric Jacobians and Global Gradient Structure}
\label{sec:gradient-structure}

The coefficient calculation of the previous section has a global
geometric interpretation.  We first recall the classical equivalence
between closed one-forms and gradient fields on simply connected
domains; see, for example, \cite{Lee2013}.

\begin{theorem}\label{thm:symmetric-gradient}
Let $\Omega\subset\R^3$ be an open and simply connected set, and let
$F=(P,Q,R)\colon\Omega\to\R^3$ be of class $C^1$.  The following
statements are equivalent:
\begin{enumerate}[label=\textup{(\roman*)}]
\item $DF(\bx)=DF(\bx)^{T}$ for every $\bx\in\Omega$;
\item $\nabla\times F=0$ in $\Omega$;
\item there exists a function $V\in C^2(\Omega)$ such that
      $F=\nabla V$.
\end{enumerate}
The potential $V$ is unique up to an additive constant on each connected
component of $\Omega$.
\end{theorem}

\begin{proof}
The Jacobian is symmetric if and only if
\[
P_y=Q_x,\qquad P_z=R_x,\qquad Q_z=R_y.
\]
Since
\[
\nabla\times F=(R_y-Q_z,\,P_z-R_x,\,Q_x-P_y),
\]
conditions \textup{(i)} and \textup{(ii)} are equivalent.  The one-form
\[
\omega=P\,\dd x+Q\,\dd y+R\,\dd z
\]
is closed precisely when these identities hold.  Since $\Omega$ is
simply connected, the Poincar\'e lemma gives a scalar function $V$ with
$\dd V=\omega$, that is, $F=\nabla V$.  Because $F$ is $C^1$, the
potential is $C^2$.  Conversely, if $F=\nabla V$, then
\[
DF=\Hess V,
\]
which is symmetric by equality of mixed partial derivatives (Clairaut-Schwarz Lemma).  Finally,
if two potentials have the same gradient, their difference has zero
gradient and is therefore constant on each connected component.
\end{proof}

For polynomial fields on $\R^3$, the potential can be chosen polynomial
without appealing to an abstract existence argument.

\begin{proposition}\label{prop:radial-potential}
Let $F\colon\R^3\to\R^3$ be a polynomial vector field of degree at most
$d$ and suppose that $DF$ is symmetric.  Then
\begin{equation}
V(\bx)=\int_0^1 \langle F(t\bx),\bx\rangle\,\dd t
\label{eq:radial-potential}
\end{equation}
is a polynomial of degree at most $d+1$, satisfies $V(0)=0$, and obeys
\[
\nabla V=F.
\]
Every other potential differs from $V$ by an additive constant.
\end{proposition}

\begin{proof}
Because $F(t\bx)$ is polynomial in $(t,\bx)$, integration with respect
to $t$ shows that $V$ is polynomial and $\deg V\le d+1$.  Write
$\bx=(x_1,x_2,x_3)$ and $F=(F_1,F_2,F_3)$.  For each $j$,
\begin{align*}
\frac{\partial V}{\partial x_j}(\bx)
&=\int_0^1\left[
F_j(t\bx)+t\sum_{i=1}^3x_i
\frac{\partial F_i}{\partial x_j}(t\bx)
\right]\dd t\\
&=\int_0^1\left[
F_j(t\bx)+t\sum_{i=1}^3x_i
\frac{\partial F_j}{\partial x_i}(t\bx)
\right]\dd t\\
&=\int_0^1\frac{\dd}{\dd t}\bigl[tF_j(t\bx)\bigr]\dd t
=F_j(\bx),
\end{align*}
where symmetry of $DF$ was used in the second equality.  Thus
$\nabla V=F$.  Uniqueness modulo constants follows from
Theorem~\ref{thm:symmetric-gradient}.
\end{proof}

\begin{corollary}
\label{cor:degree}
A quadratic polynomial vector field on $\R^3$ with symmetric Jacobian
admits a polynomial potential of degree at most three.  If the vector
field has exact degree two, then every polynomial potential has exact
degree three.
\end{corollary}

\begin{proof}
The first assertion follows from Proposition~\ref{prop:radial-potential}
with $d=2$.  For the second, let $V_m$ be the highest-degree homogeneous
part of a potential $V$.  If $m\ge1$, then the highest-degree part of
$\nabla V$ is $\nabla V_m$ and has degree $m-1$; moreover,
$\nabla V_m\not\equiv0$ because a nonconstant homogeneous polynomial
cannot have identically zero gradient.  Hence $\deg(\nabla V)=\deg V-1$.
Therefore $\deg F=2$ implies $\deg V=3$.
\end{proof}

The potential is automatically a strict Lyapunov function along every
nonstationary trajectory, with the sign convention used in this paper.

\begin{corollary}
\label{cor:monotonicity}
Let $\bx(t)$ be a solution of
\[
\dot{\bx}=\nabla V(\bx).
\]
Then
\begin{equation}
\frac{\dd}{\dd t}V(\bx(t))
=\langle\nabla V(\bx(t)),\dot{\bx}(t)\rangle
=\|\nabla V(\bx(t))\|^2\ge0.
\label{eq:monotonicity}
\end{equation}
Equality at a time $t$ holds if and only if $\bx(t)$ is an equilibrium.
Thus $V$ is strictly increasing along every nonconstant orbit.
\end{corollary}

Many texts define a gradient flow by $\dot{\bx}=-\nabla V(\bx)$, in
which case $V$ decreases along trajectories.  The two conventions are
equivalent after replacing $V$ by $-V$.  In the present manuscript the
choice $\dot{\bx}=\nabla V$ is retained because it coincides directly
with the given vector field.


\section{Cubic Polynomial Potentials and Canonical Parametrization}
\label{sec:cubic-potentials}

Let $\Pcal_{\le d}(\R^3)$ denote the vector space of real polynomials in
three variables of degree at most $d$, and define
\[
\Gcal_{\le2}(\R^3)
=
\left\{
F\in\Pcal_{\le2}(\R^3)^3:\ DF=DF^T
\right\}.
\]

\begin{proposition}
\label{prop:quotient}
The gradient operator
\[
\nabla\colon\Pcal_{\le3}(\R^3)
\longrightarrow\Gcal_{\le2}(\R^3)
\]
is linear and surjective, and
\[
\ker(\nabla)=\R.
\]
Consequently,
\begin{equation}
\Gcal_{\le2}(\R^3)
\cong
\Pcal_{\le3}(\R^3)/\R.
\label{eq:quotient}
\end{equation}
In particular,
\[
\dim\Gcal_{\le2}(\R^3)
=\binom{3+3}{3}-1
=20-1=19.
\]
\end{proposition}

\begin{proof}
Linearity is immediate.  Surjectivity follows from
Proposition~\ref{prop:radial-potential}.  The gradient of a polynomial
vanishes identically if and only if the polynomial is constant.  The
isomorphism \eqref{eq:quotient} therefore follows from the first
isomorphism theorem for vector spaces.  Finally,
$\dim\Pcal_{\le3}(\R^3)=\binom{6}{3}=20$.
\end{proof}

A coordinate-free parametrization makes the internal symmetry of the
cubic terms especially transparent.  Every $V\in\Pcal_{\le3}(\R^3)$
can be written uniquely as
\begin{equation}
V(\bx)
=\kappa+\langle\ell,\bx\rangle
+\frac12\langle A\bx,\bx\rangle
+\frac16\,T[\bx,\bx,\bx],
\label{eq:tensor-potential}
\end{equation}
where $\kappa\in\R$, $\ell\in\R^3$, $A$ is a symmetric $3\times3$
matrix, and $T=(T_{ijk})$ is a fully symmetric third-order tensor.  In
coordinates,
\begin{align}
F_i(\bx)
&=\frac{\partial V}{\partial x_i}(\bx)
=\ell_i+\sum_{j=1}^3 A_{ij}x_j
+\frac12\sum_{j,k=1}^3T_{ijk}x_jx_k,
\label{eq:tensor-vector-field}\\
(DF(\bx))_{ij}
&=\frac{\partial^2V}{\partial x_i\partial x_j}(\bx)
=A_{ij}+\sum_{k=1}^3T_{ijk}x_k.
\label{eq:tensor-hessian}
\end{align}
The dimensions of the effective data are
\[
\underbrace{3}_{\ell}
+\underbrace{6}_{A=A^T}
+\underbrace{10}_{T\text{ fully symmetric}}
=19.
\]
The omitted scalar $\kappa$ is precisely the one-dimensional kernel of
the gradient map.

A symmetric affine matrix-valued function is not automatically the
Hessian of a cubic polynomial.  Besides pointwise symmetry in the
indices $i,j$, its linear coefficients must satisfy the compatibility
relations
\[
\frac{\partial H_{ij}}{\partial x_k}
=
\frac{\partial H_{ik}}{\partial x_j}
=
\frac{\partial H_{jk}}{\partial x_i},
\]
which are equivalent to full symmetry of the tensor $T_{ijk}$ in
\eqref{eq:tensor-hessian}.  Thus the Hessian structure is more rigid
than merely being a symmetric affine matrix.

For later computations it is useful to record the potential associated
with the explicit $19$-parameter system \eqref{eq:19parameter-system}.
Up to an arbitrary additive constant $\kappa$, it is
\begin{align}
V(x,y,z)={}&\kappa+a_0x+b_0y+c_0z
+\frac12\bigl(a_1x^2+b_2y^2+c_3z^2\bigr)
+a_2xy+a_3xz+b_3yz
\nonumber\\
&+\frac{a_4}{3}x^3
+\frac{a_5}{2}x^2y
+\frac{a_6}{2}x^2z
+a_7xy^2+a_8xyz+a_9xz^2
\nonumber\\
&+\frac{b_7}{3}y^3
+\frac{b_8}{2}y^2z
+b_9yz^2
+\frac{c_9}{3}z^3.
\label{eq:explicit-potential}
\end{align}
Direct differentiation gives
\[
\nabla V=(P,Q,R),
\]
with $(P,Q,R)$ exactly as in \eqref{eq:19parameter-system}.  Conversely,
every polynomial in \eqref{eq:explicit-potential} generates a quadratic
vector field with symmetric Jacobian.

The equilibrium set is the algebraic set
\begin{equation}
\operatorname{Crit}(V)
=
\left\{\bx\in\R^3:\nabla V(\bx)=0\right\},
\label{eq:critical-set}
\end{equation}
that is, the common zero set of three quadratic polynomials.  It may be
finite or may contain positive-dimensional components.  This distinction
will be important when convergence of bounded trajectories is studied.

\subsection*{Reformulation of the conjecture}

The preceding results yield the following exact algebraic reformulation:
\begin{quote}
The class of three-dimensional quadratic polynomial vector fields with
symmetric Jacobian is precisely the class of gradient vector fields
$\dot{\bx}=\nabla V(\bx)$ generated by polynomials
$V\in\Pcal_{\le3}(\R^3)$, modulo additive constants.
\end{quote}
Accordingly, the Sprott--Zeraoulia problem may be rephrased as follows:

\medskip
\noindent\textbf{Gradient formulation of the Sprott--Zeraoulia problem.}
\emph{Determine, in a precise dynamical sense, which forms of chaotic
behavior are impossible for gradient flows generated by cubic
polynomials on $\R^3$.}
\medskip

This wording is deliberately more precise than simply saying that
``cubic gradient flows are not chaotic.''  The next sections must first
specify the relevant class of trajectories or invariant sets and then
use the monotonicity identity \eqref{eq:monotonicity}, compactness
arguments, and analytic gradient theory to establish the strongest
valid conclusion.


\section{Dynamical Consequences of the Gradient Structure}
\label{sec:dynamical-gradient}

We now turn from the algebraic description of the vector field to its
dynamical consequences.  Throughout this section,
\begin{equation}
\dot{\bx}=\nabla V(\bx),
\qquad \bx\in\R^3,
\label{eq:gradient-flow-section5}
\end{equation}
where $V\colon\R^3\to\R$ is at least of class $C^2$.  In the
Sprott--Zeraoulia class, $V$ is a polynomial of degree at most three and
is therefore real analytic.

We denote the equilibrium set by
\[
\operatorname{Crit}(V)
=
\{\bx\in\R^3:\nabla V(\bx)=0\}.
\]
Whenever a solution through $\bx_0$ is defined at time $t$, it will be
denoted by $\phi_t(\bx_0)$.

\subsection{Existence and bounded forward trajectories}

A polynomial vector field is locally Lipschitz, so every initial
condition determines a unique maximal solution.  Since quadratic vector
fields need not be globally Lipschitz, a solution may in principle
escape to infinity in finite time.  We will treat first the bounded solutions, which are more relevant to
the study of attractors and are automatically global, as stated in the next well-known result. 

\begin{lemma}\label{lem:bounded-global}
Let $F\colon\R^3\to\R^3$ be locally Lipschitz, and let
$\bx\colon[0,T_{\max})\to\R^3$ be a maximal forward solution of
$\dot{\bx}=F(\bx)$.  If $\bx([0,T_{\max}))$ is bounded, then
$T_{\max}=+\infty$.
\end{lemma}

\begin{proof}
Suppose, by contradiction, that $T_{\max}<+\infty$.  Since the image of
the solution is bounded, it is contained in a compact set $K$.  The
vector field is bounded and locally Lipschitz on a neighborhood of $K$.
The standard continuation theorem for ordinary differential equations
then allows the solution to be extended beyond $T_{\max}$, contradicting
maximality.  Hence $T_{\max}=+\infty$.
\end{proof}

Thus, from now on, every bounded forward trajectory is understood to be
defined for all $t\geq0$.

\subsection{The energy identity and immediate exclusions}

Corollary~\ref{cor:monotonicity} gives the differential identity
\[
\frac{\dd}{\dd t}V(\bx(t))
=\|\nabla V(\bx(t))\|^2
=\|\dot{\bx}(t)\|^2.
\]
Its integrated form will be used repeatedly.

\begin{proposition}
\label{prop:energy-balance}
For every solution of \eqref{eq:gradient-flow-section5} and every pair
$0\leq s\leq t$ in its interval of existence,
\begin{equation}
V(\bx(t))-V(\bx(s))
=
\int_s^t\|\nabla V(\bx(\tau))\|^2\,\dd\tau
=
\int_s^t\|\dot{\bx}(\tau)\|^2\,\dd\tau.
\label{eq:energy-balance}
\end{equation}
Consequently, $V$ is constant along an orbit if and only if that orbit
is an equilibrium.

If the forward trajectory of $\bx_0$ is bounded, then the finite limit
\begin{equation}
V_\infty(\bx_0)
:=
\lim_{t\to+\infty}V(\phi_t(\bx_0))
\label{eq:potential-limit}
\end{equation}
exists and
\begin{equation}
\int_0^{+\infty}\|\dot{\bx}(t)\|^2\,\dd t
=
V_\infty(\bx_0)-V(\bx_0)
<+\infty.
\label{eq:velocity-L2}
\end{equation}
\end{proposition}

\begin{proof}
Equation \eqref{eq:energy-balance} follows by integrating
\eqref{eq:monotonicity}.  If the trajectory is bounded, its closure is
compact, and therefore $V$ is bounded on that closure.  Since
$t\mapsto V(\bx(t))$ is nondecreasing, it converges to a finite limit.
Letting $t\to+\infty$ in \eqref{eq:energy-balance} gives
\eqref{eq:velocity-L2}.
\end{proof}

The square-integrability in \eqref{eq:velocity-L2} does not, by itself,
prove convergence of the trajectory.  A function may have derivative in
$L^2(0,+\infty)$ without having finite total variation.  The stronger
estimate needed below is supplied by the \L{}ojasiewicz gradient
inequality.

The strict monotonicity of $V$ already excludes several classical
mechanisms of recurrent dynamics.

\begin{proposition}
\label{prop:immediate-exclusions}
For the gradient system \eqref{eq:gradient-flow-section5}, the following
statements hold.
\begin{enumerate}[label=\textup{(\roman*)}]
\item There are no nonconstant periodic orbits.
\item Every positively recurrent point is an equilibrium.  More
precisely, if $t_n\to+\infty$ and
\[
\phi_{t_n}(\bx_0)\to\bx_0,
\]
then $\nabla V(\bx_0)=0$.
\item There are no nonconstant homoclinic orbits.
\item If a nonconstant complete orbit is heteroclinic from an
 equilibrium $p$ to an equilibrium $q$, then
\begin{equation}
V(p)<V(q).
\label{eq:heteroclinic-order}
\end{equation}
In particular, heteroclinic cycles are impossible.
\end{enumerate}
\end{proposition}

\begin{proof}
If $\bx(t)$ is periodic with period $T>0$, then
\eqref{eq:energy-balance} gives
\[
0=V(\bx(T))-V(\bx(0))
=\int_0^T\|\dot{\bx}(t)\|^2\,\dd t,
\]
so $\dot{\bx}\equiv0$.

Now suppose that $\phi_{t_n}(\bx_0)\to\bx_0$.  For each fixed $t>0$
and all sufficiently large $n$, $t<t_n$, and monotonicity yields
\[
V(\bx_0)
\leq V(\phi_t(\bx_0))
\leq V(\phi_{t_n}(\bx_0)).
\]
Passing to the limit as $n\to\infty$ gives
$V(\phi_t(\bx_0))=V(\bx_0)$.  Hence the orbit is stationary.

If a complete orbit is homoclinic to $p$, then
$V(\bx(t))\to V(p)$ as both $t\to-\infty$ and $t\to+\infty$.  A
nondecreasing function with equal limits at both ends must be constant,
so the orbit is stationary.  Finally, along a nonconstant heteroclinic
orbit from $p$ to $q$, the potential is strictly increasing; taking the
limits as $t\to\pm\infty$ gives \eqref{eq:heteroclinic-order}.  A cycle
would require a strict cyclic ordering of finitely many real numbers,
which is impossible.
\end{proof}

\subsection{Omega-limit sets of bounded trajectories}

For a forward complete solution through $\bx_0$, its omega-limit set is
\begin{equation}
\omega(\bx_0)
=
\left\{
\bm{y}\in\R^3:
\phi_{t_n}(\bx_0)\to\bm{y}
\text{ for some sequence }t_n\to+\infty
\right\}.
\label{eq:omega-definition}
\end{equation}
The following result is the gradient version of the conclusion suggested
by LaSalle's invariance principle \cite{LaSalle1960}.

\begin{theorem}\label{thm:omega-critical}
Let $\bx(t)=\phi_t(\bx_0)$ be a bounded forward trajectory of
\eqref{eq:gradient-flow-section5}.  Then $\omega(\bx_0)$ is nonempty,
compact, connected, and invariant.  Moreover,
\begin{equation}
\omega(\bx_0)
\subset
\operatorname{Crit}(V)
\cap
V^{-1}\bigl(V_\infty(\bx_0)\bigr).
\label{eq:omega-critical-level}
\end{equation}
Thus every accumulation point of a bounded trajectory is an equilibrium,
and all such accumulation points lie on the same critical level of the
potential.
\end{theorem}

\begin{proof}
By Lemma~\ref{lem:bounded-global}, the trajectory is defined for all
$t\geq0$.  For $T\geq0$, set
\[
K_T
=
\overline{\{\phi_t(\bx_0):t\geq T\}}.
\]
Each $K_T$ is nonempty and compact because the trajectory is bounded.
It is connected because it is the closure of the continuous image of
the connected interval $[T,+\infty)$.  The family $(K_T)_{T\geq0}$ is
nested, and
\[
\omega(\bx_0)=\bigcap_{T\geq0}K_T.
\]
It follows that $\omega(\bx_0)$ is nonempty, compact, and connected.
Its invariance is a standard consequence of continuity and uniqueness of
the flow.

Let $\bm{y}\in\omega(\bx_0)$.  There exists $t_n\to+\infty$ such that
$\phi_{t_n}(\bx_0)\to\bm{y}$.  By continuity of $V$ and
\eqref{eq:potential-limit},
\[
V(\bm{y})=V_\infty(\bx_0).
\]
Since the omega-limit set is invariant, $\phi_s(\bm{y})\in\omega(\bx_0)$
for every $s$ for which the orbit is considered.  Hence
\[
V(\phi_s(\bm{y}))=V_\infty(\bx_0)
\]
is constant in $s$.  Differentiating gives
\[
\|\nabla V(\phi_s(\bm{y}))\|^2=0.
\]
In particular, $\nabla V(\bm{y})=0$, proving
\eqref{eq:omega-critical-level}.
\end{proof}

\begin{corollary}\label{cor:isolated-convergence} Let $\bx(t)$ be a bounded forward trajectory.  If the set
\[
\operatorname{Crit}(V)
\cap V^{-1}\bigl(V_\infty(\bx_0)\bigr)
\]
is discrete, then there exists an equilibrium $p$ such that
\[
\lim_{t\to+\infty}\bx(t)=p.
\]
In particular, this conclusion holds whenever all equilibria are
isolated.
\end{corollary}

\begin{proof}
By Theorem~\ref{thm:omega-critical}, the omega-limit set is a connected
subset of a discrete set.  Therefore it consists of a single point,
say $\omega(\bx_0)=\{p\}$.  A bounded trajectory whose omega-limit set
is a singleton converges to that point.
\end{proof}

The preceding argument does not settle the case in which the critical
set contains curves or surfaces.  For cubic polynomial potentials,
this possible degeneracy is resolved by analyticity.

\subsection{The \L{}ojasiewicz gradient inequality and convergence}

We recall the local gradient inequality of \L{}ojasiewicz.  It is the
key analytic ingredient in the asymptotic theory of gradient
trajectories \cite{Lojasiewicz1984,Kurdyka1998}.

\begin{theorem}\label{thm:lojasiewicz-inequality}
Let $V$ be real analytic in a neighborhood of a critical point $p$.
Then there exist a neighborhood $U$ of $p$, a constant $C>0$, and an
exponent $\theta\in(0,1)$ such that
\begin{equation}
|V(\bx)-V(p)|^\theta
\leq C\|\nabla V(\bx)\|,
\qquad \bx\in U.
\label{eq:lojasiewicz-inequality}
\end{equation}
\end{theorem}

For completeness, we now show in detail how this local inequality turns
the qualitative description of the omega-limit set into convergence to
one point.

\begin{theorem}
\label{thm:analytic-gradient-convergence}
Let $V\colon\R^3\to\R$ be real analytic, and let
$\bx(t)=\phi_t(\bx_0)$ be a bounded forward solution of
\eqref{eq:gradient-flow-section5}.  Then there exists a single critical
point $p\in\operatorname{Crit}(V)$ such that
\begin{equation}
\lim_{t\to+\infty}\bx(t)=p.
\label{eq:trajectory-convergence}
\end{equation}
Moreover, the trajectory has finite length:
\begin{equation}
\int_0^{+\infty}\|\dot{\bx}(t)\|\,\dd t<+\infty.
\label{eq:finite-length}
\end{equation}
\end{theorem}

\begin{proof}
Let
\[
c=V_\infty(\bx_0),
\qquad
\Gamma=\omega(\bx_0).
\]
By Theorem~\ref{thm:omega-critical}, $\Gamma$ is a compact subset of
$\operatorname{Crit}(V)\cap V^{-1}(c)$.

For every $p\in\Gamma$, Theorem~\ref{thm:lojasiewicz-inequality}
provides a neighborhood $U_p$, a constant $C_p>0$, and an exponent
$\theta_p\in(0,1)$ such that
\[
|V(\bx)-c|^{\theta_p}
\leq C_p\|\nabla V(\bx)\|,
\qquad \bx\in U_p.
\]
By compactness, finitely many of these neighborhoods cover $\Gamma$.
After shrinking their union so that $|V-c|\leq1$, taking the largest of
the finitely many exponents and increasing the constant if necessary,
we obtain a neighborhood $U$ of $\Gamma$, a constant $C>0$, and one
exponent $\theta\in(0,1)$ such that
\begin{equation}
|V(\bx)-c|^\theta
\leq C\|\nabla V(\bx)\|,
\qquad \bx\in U.
\label{eq:uniform-lojasiewicz}
\end{equation}

For every bounded forward trajectory,
\[
\operatorname{dist}(\bx(t),\Gamma)\longrightarrow0
\qquad\text{as }t\to+\infty.
\]
Indeed, otherwise a sequence of points of the trajectory would remain a
fixed positive distance from $\Gamma$; boundedness would produce an
accumulation point outside $\Gamma$, contradicting the definition of the
omega-limit set.  Hence there exists $t_0$ such that $\bx(t)\in U$ for
all $t\geq t_0$.

Define the nonnegative energy gap
\[
e(t)=c-V(\bx(t)).
\]
Since $V(\bx(t))$ increases to $c$, we have $e(t)\to 0$, and
\begin{equation}
e'(t)=-\|\nabla V(\bx(t))\|^2.
\label{eq:energy-gap-derivative}
\end{equation}
If $e(t)$ becomes zero at a finite time, monotonicity forces the
trajectory to be stationary thereafter, and the conclusion follows.
Assume therefore that $e(t)>0$ for all sufficiently large $t$.
Using \eqref{eq:uniform-lojasiewicz}, we obtain
\begin{align*}
-\frac{\dd}{\dd t}e(t)^{1-\theta}
&=(1-\theta)e(t)^{-\theta}
  \|\nabla V(\bx(t))\|^2\\
&\geq
\frac{1-\theta}{C}\|\nabla V(\bx(t))\|
=
\frac{1-\theta}{C}\|\dot{\bx}(t)\|.
\end{align*}
Integration from $t$ to $T>t$ gives
\[
\int_t^T\|\dot{\bx}(s)\|\,\dd s
\leq
\frac{C}{1-\theta}
\left(e(t)^{1-\theta}-e(T)^{1-\theta}\right).
\]
Letting $T\to+\infty$ yields
\begin{equation}
\int_t^{+\infty}\|\dot{\bx}(s)\|\,\dd s
\leq
\frac{C}{1-\theta}e(t)^{1-\theta}.
\label{eq:tail-length-estimate}
\end{equation}
Thus the tail has finite length, and adding the finite initial segment
proves \eqref{eq:finite-length}.  Moreover, for $t_2>t_1\geq t_0$,
\[
\|\bx(t_2)-\bx(t_1)\|
\leq
\int_{t_1}^{t_2}\|\dot{\bx}(s)\|\,\dd s.
\]
The right-hand side tends to zero as $t_1,t_2\to+\infty$, so
$\bx(t)$ is a Cauchy curve and converges to a point $p$.  Necessarily
$p\in\Gamma\subset\operatorname{Crit}(V)$.
\end{proof}

\begin{corollary}\label{cor:cubic-bounded-convergence}
Every bounded forward trajectory of a quadratic polynomial vector field
on $\R^3$ with symmetric Jacobian converges to a single equilibrium and
has finite length.
\end{corollary}

\begin{proof}
By the results of Sections~\ref{sec:gradient-structure} and
\ref{sec:cubic-potentials}, the vector field is the gradient of a
polynomial potential of degree at most three.  Polynomial functions are
real analytic, so Theorem~\ref{thm:analytic-gradient-convergence}
applies.
\end{proof}

\subsection{Finite accumulation points and the bounded--escape dichotomy}

The boundedness hypothesis in
Theorem~\ref{thm:analytic-gradient-convergence} can be replaced by the
existence of one finite accumulation point.  This refinement is useful
because it separates all maximal trajectories into those converging in
$\R^3$ and those genuinely escaping to infinity.

\begin{theorem}\label{thm:finite-accumulation-convergence}
Let $V\colon\R^3\to\R$ be real analytic, and let
$\bx\colon[0,T_{\max})\to\R^3$ be a maximal forward solution of
$\dot\bx=\nabla V(\bx)$.  Suppose that there exist times
$t_n\to  T_{\max}$ and a point $p\in\R^3$ such that
\[
 \bx(t_n)\longrightarrow p.
\]
Then $T_{\max}=+\infty$, the point $p$ is critical, and
\[
 \bx(t)\longrightarrow p
 \qquad\text{as }t\to+\infty.
\]
Moreover, the tail of the trajectory has finite length.
\end{theorem}

\begin{proof}
If $T_{\max}<+\infty$, the continuation theorem implies that a maximal
solution must leave every compact subset of $\R^3$ as
$t\to T_{\max}$.  This contradicts the convergent subsequence
$\bx(t_n)\to p$.  Hence $T_{\max}=+\infty$.

The function $t\mapsto V(\bx(t))$ is nondecreasing.  Since
$V(\bx(t_n))\to V(p)$, monotonicity implies that
\begin{equation}
 \lim_{t\to+\infty}V(\bx(t))=V(p)=:c.
 \label{eq:finite-accumulation-energy-limit}
\end{equation}
For every sufficiently small fixed $s\geq0$, continuous dependence on initial conditions
gives
\[
 \bx(t_n+s)=\phi_s(\bx(t_n))\longrightarrow\phi_s(p).
\]
On the other hand, \eqref{eq:finite-accumulation-energy-limit} gives
$V(\bx(t_n+s))\to c$.  Therefore
$V(\phi_s(p))=c$ for every $s\geq0$.  Differentiating at $s=0$ yields
\[
 \|\nabla V(p)\|^2=0,
\]
so $p$ is critical.

Choose a neighborhood $U$ of $p$ on which the \L{}ojasiewicz inequality
\[
 |V(\bx)-c|^\theta\leq C\|\nabla V(\bx)\|
\]
holds.  Take radii $0<r_1<r_2$ such that
$\overline{B(p,r_2)}\subset U$.  Put
$e(t)=c-V(\bx(t))$.  For $n$ sufficiently large,
$\bx(t_n)\in B(p,r_1)$ and
\begin{equation}
 \frac{C}{1-\theta}e(t_n)^{1-\theta}<r_2-r_1.
 \label{eq:lojasiewicz-trapping-choice}
\end{equation}
Suppose that the orbit exits $B(p,r_2)$ after time $t_n$, and let
$T_n>t_n$ be its first exit time.  The same calculation as in
\eqref{eq:tail-length-estimate}, applied on $[t_n,T_n]$, gives
\[
 \int_{t_n}^{T_n}\|\dot\bx(s)\|\dd s
 \leq \frac{C}{1-\theta}e(t_n)^{1-\theta}
 <r_2-r_1.
\]
This is impossible because a curve starting in $B(p,r_1)$ and reaching
$\partial B(p,r_2)$ has length at least $r_2-r_1$.  Thus the trajectory
remains in $B(p,r_2)$ for all sufficiently large time.  The
\L{}ojasiewicz estimate then gives finite tail length exactly as in the
proof of Theorem~\ref{thm:analytic-gradient-convergence}, and hence the
trajectory converges.  Its convergent subsequence already tends to $p$,
so the limit is $p$.
\end{proof}

\begin{corollary}\label{cor:finite-or-infinity-dichotomy}
Let $F=\nabla V$ be a polynomial gradient vector field on $\R^3$, and
let $\bx\colon[0,T_{\max})\to\R^3$ be a maximal forward solution.  Then
exactly one of the following alternatives holds:
\begin{enumerate}[label=\textup{(\roman*)}]
\item $T_{\max}=+\infty$ and $\bx(t)$ converges to a finite equilibrium;
\item
\begin{equation}
 \|\bx(t)\|\longrightarrow+\infty
 \qquad\text{as }t\to T_{\max},
 \label{eq:escape-to-infinity}
\end{equation}
where $T_{\max}$ may be finite or infinite.
\end{enumerate}
\end{corollary}

\begin{proof}
If \eqref{eq:escape-to-infinity} fails, there is a sequence
$t_n\uparrow T_{\max}$ for which $\bx(t_n)$ remains bounded.  Passing to
a subsequence gives a finite accumulation point.  Theorem
\ref{thm:finite-accumulation-convergence} then yields the first
alternative.  The two alternatives are mutually exclusive.
\end{proof}

\subsection{Nonwandering dynamics and entropy on compact invariant sets}

The convergence theorem has strong consequences for compact invariant
sets.  Recall that a point $\bx$ of a compact invariant set $K$ is
\emph{nonwandering} if, for every neighborhood $U$ of $\bx$ in $K$ and
every $T>0$, there exists $t>T$ such that
$\phi_t(U)\cap U\neq\varnothing$.

\begin{proposition}\label{prop:nonwandering-critical}
Let $K\subset\R^3$ be a compact invariant set of
\eqref{eq:gradient-flow-section5}.  Then
\begin{equation}
\operatorname{NW}(\phi|_K)
=
K\cap\operatorname{Crit}(V).
\label{eq:nonwandering-critical}
\end{equation}
\end{proposition}

\begin{proof}
Every equilibrium is nonwandering.  Conversely, let $p\in K$ be a
noncritical point.  Since
\[
\left.\frac{\dd}{\dd t}V(\phi_t(p))\right|_{t=0}
=\|\nabla V(p)\|^2>0,
\]
there exists $\tau>0$ such that
$V(\phi_\tau(p))>V(p)$.  By continuity of the flow and of $V$, one can
choose a sufficiently small neighborhood $U$ of $p$ in $K$ satisfying
\[
\inf_{\bm{y}\in U}V(\phi_\tau(\bm{y}))
>
\sup_{\bm{y}\in U}V(\bm{y}).
\]
For every $t\geq\tau$ and $\bm{y}\in U$, monotonicity gives
\[
V(\phi_t(\bm{y}))
\geq V(\phi_\tau(\bm{y}))
>
\sup_{\bm{z}\in U}V(\bm{z}).
\]
Therefore $\phi_t(U)\cap U=\varnothing$ for all $t\geq\tau$, and $p$
is wandering.
\end{proof}

We next formulate an entropy consequence.  We use the variational
principle and Poincar\'e recurrence in their standard compact-metric
forms; see \cite{Walters1982}.

\begin{theorem}\label{thm:zero-entropy}
Let $K\subset\R^3$ be a compact invariant set of a gradient flow
\eqref{eq:gradient-flow-section5}.  Then, for every $\tau\neq0$, the
time-$\tau$ map satisfies
\begin{equation}
h_{\mathrm{top}}\bigl(\phi_\tau|_K\bigr)=0.
\label{eq:zero-topological-entropy}
\end{equation}
In particular, no compact invariant set of such a flow can support
dynamics with positive topological entropy.
\end{theorem}

\begin{proof}
It is enough to consider $\tau>0$, since a homeomorphism and its inverse
have the same topological entropy.  Let
$f=\phi_\tau|_K$, and let $\mu$ be any $f$-invariant Borel probability
measure.  By Poincar\'e recurrence, $\mu$-almost every point is recurrent
under $f$.  Proposition~\ref{prop:immediate-exclusions} implies that
every such point is an equilibrium.  Hence $\mu$ is supported on
$K\cap\operatorname{Crit}(V)$, where $f$ is the identity.  Therefore
$h_\mu(f)=0$ for every invariant probability measure $\mu$.  The
variational principle gives
\[
h_{\mathrm{top}}(f)
=
\sup_\mu h_\mu(f)
=0.
\]
\end{proof}

\begin{corollary}\label{cor:no-li-yorke}
Assume that $V$ is real analytic, and let $K$ be compact and invariant.
For any $\bx,\bm{y}\in K$, the limit
\[
\lim_{t\to+\infty}
\|\phi_t(\bx)-\phi_t(\bm{y})\|
\]
exists.  Consequently, $K$ contains no Li--Yorke scrambled pair.
\end{corollary}

\begin{proof}
By Theorem~\ref{thm:analytic-gradient-convergence}, there exist
equilibria $p$ and $q$ such that
$\phi_t(\bx)\to p$ and $\phi_t(\bm{y})\to q$.  Hence
\[
\|\phi_t(\bx)-\phi_t(\bm{y})\|
\longrightarrow\|p-q\|.
\]
The liminf and the limsup of the mutual distance are therefore equal,
which excludes a Li--Yorke pair.
\end{proof}

\subsection{Main consequence for the Sprott--Zeraoulia class}

The results of this section may be summarized as follows.

\begin{theorem}\label{thm:main-gradient-dynamics}
Let $F\colon\R^3\to\R^3$ be a polynomial vector field of degree at most
two with symmetric Jacobian matrix.  Then:
\begin{enumerate}[label=\textup{(\roman*)}]
\item every maximal forward trajectory either converges to a finite
equilibrium or escapes to infinity; every bounded forward trajectory is
global, has finite length, and converges to a single equilibrium;
\item every positively recurrent point is an equilibrium;
\item nonconstant periodic or homoclinic orbits and heteroclinic cycles
are impossible;
\item on every compact invariant set, the nonwandering set consists
exactly of the equilibria;
\item every time map restricted to a compact invariant set has zero
topological entropy;
\item no compact invariant set contains a Li--Yorke scrambled pair.
\end{enumerate}
\end{theorem}

\begin{proof}
The gradient representation and analyticity follow from
Sections~\ref{sec:gradient-structure} and \ref{sec:cubic-potentials}.
The trajectory dichotomy follows from
Corollary~\ref{cor:finite-or-infinity-dichotomy}; the bounded convergence
statement is Corollary~\ref{cor:cubic-bounded-convergence}.  The remaining
conclusions are
Proposition~\ref{prop:immediate-exclusions},
Proposition~\ref{prop:nonwandering-critical},
Theorem~\ref{thm:zero-entropy}, and
Corollary~\ref{cor:no-li-yorke}.
\end{proof}

It is worth noting that Theorem~\ref{thm:main-gradient-dynamics} does not assert that every
compact invariant set consists only of equilibria.  Such a set may also
contain heteroclinic connections ordered by strictly increasing values
of $V$.  What the theorem shows is that all recurrent dynamics is
concentrated on the equilibrium set and that every nonstationary bounded
orbit converges to an equilibrium in forward time.

Also, a positive Lyapunov exponent, taken in isolation, is not equivalent to
chaos.  For example, an unstable equilibrium may have positive
linearized exponents, and an exceptional orbit converging to a saddle
may still possess an expanding transverse direction.  The conclusions
above avoid this ambiguity: convergence of every bounded orbit, absence
of nontrivial recurrence, and zero entropy on compact invariant sets are
intrinsic dynamical statements.

The theorem gives a rigorous negative answer to the Sprott--Zeraoulia
problem for the standard compact-invariant formulations of continuous-
time chaos in $\R^3$, including strange attractors with recurrent
dynamics, positive topological entropy, and Li--Yorke scrambling.
Corollary~\ref{cor:finite-or-infinity-dichotomy} shows that every
remaining maximal trajectory genuinely escapes.  Its asymptotic
directions and the dynamics on the compactifying sphere are analyzed in
Section~\ref{sec:infinity}.


\section{Exclusion of the Standard Notions of Chaos}
\label{sec:exclusion-chaos}

The word \emph{chaos} is used in the literature with several inequivalent
meanings.  Some definitions emphasize topological transitivity and
sensitivity, whereas others are formulated in terms of scrambled pairs,
entropy, symbolic dynamics, invariant measures, or strange attractors.
There is no unique and universally accepted definition of chaotic behavior,
reflecting the fact that chaos has several distinct mathematical
manifestations.  This issue has received considerable attention; see, for
example, the series of papers by Brown and Chua
\cite{chua1,chua2,chua3} and the discussion by Sander and Yorke
\cite{yorke}.  As emphasized in \cite{yorke},

\begin{quote}``\textit{In fact, chaos cannot now be satisfactorily defined mathematically using a single definition, not because chaos is not a single concept, but because chaos has many manifestations in many different situations.'' ... ``For the definition to be useful, the determination of chaos should depend on the viewpoint of the investigator. It must be phrased in terms of the information that is available to the scientist in question.''}
\end{quote}

Therefore, a rigorous resolution of the Sprott--Zeraoulia conjecture should
not rely on one informal interpretation of chaotic behavior.  Instead, the
structural results obtained in Section~\ref{sec:dynamical-gradient} will
be compared with the principal mathematical definitions used in the
literature; see, among others,
\cite{Devaney1989,AuslanderYorke1980,LiYorke1975,
BanksEtAl1992,AdlerKonheimMcAndrew1965,SchweizerSmital1994}.

Throughout this section, $F=\nabla V$ is a polynomial vector field of
degree at most two on $\R^3$, the potential $V$ is a polynomial of degree
at most three, and $K\subset\R^3$ is a compact invariant set.  The flow
restricted to $K$ is complete.  For a fixed $\tau>0$, we write
\begin{equation}
 f_\tau=\phi_\tau|_K\colon K\longrightarrow K.
 \label{eq:time-map-section6}
\end{equation}
The map $f_\tau$ is a homeomorphism, and every point of $K$ converges,
under positive iteration, to a fixed point of $f_\tau$ by
Corollary~\ref{cor:cubic-bounded-convergence}.

\subsection{Absence of topological transitivity and mixing}

Topological transitivity is one of the most persistent ingredients in
mathematical definitions of chaos.  For a continuous map $f$ on a
compact metric space $K$, it means that for every pair of nonempty open
sets $U,W\subset K$ there exists $n\geq0$ such that
$f^n(U)\cap W\neq\varnothing$.

The potential itself gives a direct obstruction to this property.

\begin{theorem}\label{thm:no-transitivity}
Let $K$ be a compact invariant set of the gradient flow
$\dot{\bx}=\nabla V(\bx)$.  If $K$ contains more than one point, then,
for every $\tau>0$, the time map $f_\tau$ is not topologically
transitive on $K$.
\end{theorem}

\begin{proof}
There are two cases.

Assume first that $V|_K$ is not constant.  Choose $a,b\in K$ such that
$V(a)>V(b)$, and choose real numbers $\alpha$ and $\beta$ satisfying
\[
 V(a)>\alpha>\beta>V(b).
\]
The sets
\[
 U=\{\bx\in K:V(\bx)>\alpha\},
 \qquad
 W=\{\bx\in K:V(\bx)<\beta\}
\]
are nonempty and open in the relative topology of $K$.  Since $V$ is
nondecreasing along positive trajectories, for every $\bx\in U$ and
every $n\geq0$,
\[
 V(f_\tau^n(\bx))\geq V(\bx)>\alpha>\beta.
\]
Hence $f_\tau^n(U)\cap W=\varnothing$ for every $n\geq0$, contradicting
transitivity.

Assume now that $V|_K$ is constant.  Since $K$ is invariant, for every
$\bx\in K$ the function $t\mapsto V(\phi_t(\bx))$ is constant.  The
identity
\[
 \frac{\dd}{\dd t}V(\phi_t(\bx))
 =\|\nabla V(\phi_t(\bx))\|^2
\]
therefore implies $\nabla V(\bx)=0$.  Thus every point of $K$ is an
equilibrium and $f_\tau$ is the identity map on $K$.  If $K$ has more
than one point, two disjoint nonempty relatively open subsets of $K$
can be chosen, and the identity map cannot send one of them into the
other.  Hence it is not transitive.
\end{proof}

\begin{corollary}\label{cor:no-mixing}
On a compact invariant set containing more than one point, no time map
$f_\tau$ is topologically mixing, weakly mixing, or exact.  In
particular, the restricted flow admits no nontrivial compact invariant
set on which the forward dynamics has a dense orbit.
\end{corollary}

\begin{proof}
Each of the stated mixing properties implies topological transitivity.
The conclusion follows from Theorem~\ref{thm:no-transitivity}.  A dense
positive orbit would also imply transitivity on a compact metric space,
and is therefore impossible.
\end{proof}

\subsection{Devaney and Auslander--Yorke chaos}

Devaney's definition combines topological transitivity, density of
periodic points, and sensitive dependence on initial conditions
\cite{Devaney1989}.  On an infinite metric space, transitivity together
with density of periodic points already implies sensitivity
\cite{BanksEtAl1992}.  Auslander and Yorke introduced a related notion
based on transitivity and sensitivity \cite{AuslanderYorke1980}.

\begin{definition}Let $f\colon K\to K$ be continuous on a compact metric space.
\begin{enumerate}[label=\textup{(\roman*)}]
\item The map $f$ is \emph{chaotic in the sense of Devaney} if it is
 topologically transitive, its periodic points are dense in $K$, and it
 has sensitive dependence on initial conditions.
\item The map $f$ is \emph{chaotic in the sense of Auslander--Yorke} if
 it is topologically transitive and sensitive.
\end{enumerate}
\end{definition}

\begin{theorem}\label{thm:no-devaney-ay}
Let $K$ be a compact invariant set of a quadratic polynomial vector
field on $\R^3$ with symmetric Jacobian.  For every $\tau>0$, the time
map $f_\tau$ is neither Devaney chaotic nor Auslander--Yorke chaotic on
any subset $K$ containing more than one point.
\end{theorem}

\begin{proof}
Both definitions require topological transitivity, which is excluded by
Theorem~\ref{thm:no-transitivity}.

There is also a second obstruction specific to Devaney chaos.  A point
periodic under $f_\tau$ either belongs to a periodic orbit of the flow or
is an equilibrium.  Proposition~\ref{prop:immediate-exclusions} excludes
nonconstant periodic orbits.  Therefore
\[
 \operatorname{Per}(f_\tau)=K\cap\operatorname{Crit}(V).
\]
If these periodic points were dense in $K$, their closedness would imply
$K\subset\operatorname{Crit}(V)$, and $f_\tau$ would be the identity on
$K$.  Such a map is transitive only when $K$ is a singleton.
\end{proof}

\subsection{Li--Yorke, mean Li--Yorke, and distributional chaos}

Li and Yorke introduced the terminology of chaos through the existence
of pairs whose trajectories become arbitrarily close and repeatedly
separate by a definite amount \cite{LiYorke1975}.  For a continuous map
$f$ on a metric space $(K,d)$, a pair of distinct points $(x,y)$ is a
Li--Yorke pair if
\begin{equation}
 \liminf_{n\to\infty}d(f^n(x),f^n(y))=0,
 \qquad
 \limsup_{n\to\infty}d(f^n(x),f^n(y))>0.
 \label{eq:li-yorke-pair}
\end{equation}
A system is Li--Yorke chaotic if it contains an uncountable scrambled
set, meaning that every distinct pair in that set satisfies
\eqref{eq:li-yorke-pair}.  This notion has been extended to general
compact dynamical systems; see \cite{BlanchardGlasnerKolyadaMaass2002}.

The convergence theorem from Section~\ref{sec:dynamical-gradient} gives
more than the absence of an uncountable scrambled set: it excludes even
a single scrambled pair.

\begin{theorem}\label{thm:pairwise-distance-convergence}
Let $K$ be a compact invariant set and let $\tau>0$.  For every
$x,y\in K$ there exist equilibria $p,q\in K$ such that
\[
 f_\tau^n(x)\longrightarrow p,
 \qquad
 f_\tau^n(y)\longrightarrow q.
\]
Consequently,
\begin{equation}
 \lim_{n\to\infty}d(f_\tau^n(x),f_\tau^n(y))=d(p,q).
 \label{eq:distance-limit-discrete}
\end{equation}
The same conclusion holds in continuous time:
\[
 \lim_{t\to+\infty}d(\phi_t(x),\phi_t(y))=d(p,q).
\]
\end{theorem}

\begin{proof}
Every orbit in $K$ is bounded.  By
Corollary~\ref{cor:cubic-bounded-convergence}, it converges to a single
equilibrium.  Continuity of the distance function gives
\eqref{eq:distance-limit-discrete} and its continuous-time analogue.
\end{proof}

\begin{corollary}\label{cor:no-li-yorke-mean}
No time map on a compact invariant set contains a Li--Yorke pair.  It
also contains no mean Li--Yorke pair, since the Ces\`aro averages
\[
 \frac1N\sum_{n=0}^{N-1}
 d(f_\tau^n(x),f_\tau^n(y))
\]
converge to the same limit $d(p,q)$.
\end{corollary}

Distributional chaos refines the Li--Yorke idea by considering the
asymptotic frequency with which two trajectories are close
\cite{SchweizerSmital1994}.  For $s>0$, define
\begin{align*}
 \Phi_{xy}(s)
 &=\liminf_{N\to\infty}
 \frac1N\#\bigl\{0\leq n<N:
 d(f_\tau^n(x),f_\tau^n(y))<s\bigr\},\\
 \Phi^*_{xy}(s)
 &=\limsup_{N\to\infty}
 \frac1N\#\bigl\{0\leq n<N:
 d(f_\tau^n(x),f_\tau^n(y))<s\bigr\}.
\end{align*}
The standard types DC1, DC2, and DC3 require a nontrivial discrepancy
between these lower and upper distribution functions
\cite{BalibreaSmitalStefankova2005}.

\begin{theorem}\label{thm:no-distributional-chaos}
For any $x,y\in K$, let $L=d(p,q)$ be the limit in
\eqref{eq:distance-limit-discrete}.  Then
\[
 \Phi_{xy}(s)=\Phi^*_{xy}(s)=0
 \quad\text{for }0<s<L,
\]
and
\[
 \Phi_{xy}(s)=\Phi^*_{xy}(s)=1
 \quad\text{for }s>L.
\]
Thus a possible discrepancy can occur only at the single threshold
$s=L$, and no pair is distributionally scrambled in any of the standard
senses DC1, DC2, or DC3.
\end{theorem}

\begin{proof}
If $s<L$, the convergence
\[
 d(f_\tau^n(x),f_\tau^n(y))\longrightarrow L
\]
implies that the inequality
$d(f_\tau^n(x),f_\tau^n(y))<s$ is eventually false.  Its asymptotic
frequency is therefore zero.  If $s>L$, the same inequality is
eventually true, and its asymptotic frequency is one.  The definitions
of DC1 and DC2 require a discrepancy on positive scales, while DC3
requires such a discrepancy throughout a nondegenerate interval.  A
single possible threshold cannot satisfy any of these conditions.
\end{proof}

\subsection{Entropy, invariant measures, and symbolic dynamics}

Topological entropy, introduced by Adler, Konheim, and McAndrew
\cite{AdlerKonheimMcAndrew1965}, measures the exponential growth of
orbit complexity.  Theorem~\ref{thm:zero-entropy} already shows that
all time maps have zero topological entropy on compact invariant sets.
We now refine that conclusion at the level of invariant measures.

\begin{theorem}\label{thm:invariant-measures-equilibria}
Let $\mu$ be an $f_\tau$-invariant Borel probability measure on $K$.
Then
\begin{equation}
 \operatorname{supp}\mu
 \subset K\cap\operatorname{Crit}(V).
 \label{eq:measure-support-critical}
\end{equation}
Every ergodic invariant probability measure is a Dirac measure
$\delta_p$ concentrated at one equilibrium $p\in K$.  Consequently,
\begin{equation}
 h_\mu(f_\tau)=0
 \label{eq:metric-entropy-zero}
\end{equation}
for every invariant probability measure $\mu$.
\end{theorem}

\begin{proof}
By Poincar\'e recurrence, $\mu$-almost every point is recurrent under
$f_\tau$.  Proposition~\ref{prop:immediate-exclusions} implies that all
such points are equilibria.  Since $K\cap\operatorname{Crit}(V)$ is
closed, \eqref{eq:measure-support-critical} follows.

On the support of $\mu$, the map $f_\tau$ is the identity.  If an
ergodic measure for the identity map had support containing two distinct
points, one could choose a Borel set of measure strictly between zero
and one; every Borel set is invariant under the identity, contradicting
ergodicity.  Hence every ergodic measure is a Dirac mass.  Finally, the
identity map has zero measure-theoretic entropy, proving
\eqref{eq:metric-entropy-zero}.
\end{proof}

\begin{corollary}\label{cor:no-entropy-chaos}
No compact invariant set supports positive topological entropy, positive
metric entropy, or an ergodic non-atomic invariant probability measure.
In particular, the flow has no chaotic invariant measure in the
entropy-theoretic sense.
\end{corollary}

A standard mechanism for chaos is the presence of a Smale horseshoe or,
more generally, a compact invariant subsystem conjugate or
semiconjugate to a full shift on at least two symbols
\cite{Smale1967}.  Such a shift has positive topological entropy.

\begin{corollary}\label{cor:no-horseshoe}
No time map $f_\tau$ possesses a compact invariant subset on which it is
conjugate to a full shift on two or more symbols.  More generally, no
compact invariant subsystem can factor onto such a shift.  Therefore,
Smale horseshoes and the usual symbolic mechanisms for chaotic dynamics
are impossible.
\end{corollary}

\begin{proof}
A full shift on $m\geq2$ symbols has entropy $\log m>0$.  Entropy is
preserved by conjugacy, does not increase when restricting to a
subsystem, and cannot increase when passing from a system to one of its
factors.  Any of the stated symbolic structures would therefore force
positive entropy for $f_\tau|_K$, contradicting
Theorem~\ref{thm:zero-entropy}.
\end{proof}

\subsection{Chain recurrence and gradient-like dynamics}

Conley's theory isolates chain recurrence as the part of a dynamical
system that remains after all gradient-like motion has been removed
\cite{Conley1978}.  A strict Lyapunov function is constant on chain
recurrent components and strictly monotone outside the chain recurrent
set.

\begin{proposition}\label{prop:chain-recurrent-critical}
For the restriction of the flow to a compact invariant set $K$, every
chain recurrent point belongs to $K\cap\operatorname{Crit}(V)$.  Thus
the chain recurrent dynamics is the identity dynamics on a compact
subset of equilibria.
\end{proposition}

\begin{proof}
The potential $V|_K$ is a strict Lyapunov function: it is strictly
increasing along every nonstationary orbit and constant precisely on
equilibria.  The fundamental theorem of Conley theory implies that no
point at which a Lyapunov function is strict can be chain recurrent.
Hence the chain recurrent set is contained in
$K\cap\operatorname{Crit}(V)$.  The reverse inclusion is immediate,
since every equilibrium is chain recurrent.
\end{proof}

If the equilibrium set contains a connected continuum, that continuum
may be chain connected through arbitrarily small jumps.  This does not
represent dynamical chaos: the actual flow is the identity on the
continuum.  The distinction between true trajectories and pseudo-orbits
is essential here.

\subsection{Attractors and the absence of strange attractors}

The expression \emph{strange attractor} has no single universally
accepted definition.  In the classical literature it is associated with
nonperiodic recurrent motion, transitivity, sensitive dependence,
expansion, symbolic dynamics, or positive entropy
\cite{RuelleTakens1971,Milnor1985}.  Every one of these dynamical
features has already been excluded on compact invariant sets.

\begin{proposition}\label{prop:compact-attractors}
Let $A$ be a compact invariant attracting set for the gradient flow.
Then:
\begin{enumerate}[label=\textup{(\roman*)}]
\item every trajectory in $A$ converges to an equilibrium in $A$;
\item $\operatorname{NW}(\phi|_A)=A\cap\operatorname{Crit}(V)$;
\item every time map on $A$ has zero topological entropy;
\item $A$ contains no Li--Yorke or distributionally scrambled pair;
\item if the flow on $A$ is topologically transitive, then $A$ is a
 singleton equilibrium.
\end{enumerate}
Thus a compact attracting set may contain equilibria and
heteroclinic connections, but it cannot be transitive, recurrently
nonstationary, scrambled, or of positive entropy.
\end{proposition}

\begin{proof}
The first four conclusions follow from
Corollary~\ref{cor:cubic-bounded-convergence},
Proposition~\ref{prop:nonwandering-critical},
Theorem~\ref{thm:zero-entropy}, and
Theorems~\ref{thm:pairwise-distance-convergence}--
\ref{thm:no-distributional-chaos}.  The last conclusion is
Theorem~\ref{thm:no-transitivity}.
\end{proof}

\subsection{Lyapunov exponents: what is and is not excluded}

Lyapunov exponents describe infinitesimal growth rates of tangent
vectors and are formalized by the multiplicative ergodic theorem of
Oseledets \cite{Oseledets1968}.  A positive largest Lyapunov exponent is
often used as a numerical indicator of chaos.  Taken alone, however, it
is not equivalent to chaotic dynamics.

\begin{proposition}\label{prop:lyapunov-equilibria}
Let $\mu$ be an ergodic invariant probability measure on a compact
invariant set $K$.  Then $\mu=\delta_p$ for some equilibrium $p$.  The
Lyapunov exponents of the flow with respect to $\mu$ are the eigenvalues
of
\[
 DF(p)=\operatorname{Hess}V(p),
\]
and are therefore real.  They may include positive values when $p$ is
unstable, even though the measure is supported on a single fixed point
and has zero entropy.
\end{proposition}

\begin{proof}
The classification $\mu=\delta_p$ follows from
Theorem~\ref{thm:invariant-measures-equilibria}.  Along the stationary
orbit $p$, the variational equation is
\[
 \dot\xi=DF(p)\xi.
\]
Since $DF(p)$ is symmetric, it is orthogonally diagonalizable with real
eigenvalues.  The exponential growth rates are precisely those
eigenvalues.  The entropy is zero by
Theorem~\ref{thm:invariant-measures-equilibria}.
\end{proof}

\subsection{A precise resolution for bounded dynamics}

We can now combine the gradient representation from
Sections~\ref{sec:gradient-structure}--\ref{sec:cubic-potentials}, the
convergence theorem from Section~\ref{sec:dynamical-gradient}, and the
exclusions proved above.

\begin{theorem}[Exclusion theorem for the Sprott--Zeraoulia class]
\label{thm:exclusion-standard-chaos}
Let
\[
 \dot{\bx}=F(\bx),
 \qquad \bx\in\R^3,
\]
be a polynomial differential system of degree at most two whose
Jacobian matrix is symmetric at every point.  Then every bounded forward
trajectory converges to a single equilibrium.  Moreover, on every
compact invariant set and for every nonzero time map, the system has:
\begin{enumerate}[label=\textup{(\roman*)}]
\item no nonconstant periodic orbit, homoclinic orbit, or heteroclinic
 cycle;
\item no nonstationary recurrent or nonwandering dynamics;
\item no topological transitivity, weak mixing, or topological mixing on
 a set containing more than one point;
\item no Devaney or Auslander--Yorke chaos;
\item no Li--Yorke, mean Li--Yorke, or distributional chaos;
\item zero topological entropy and zero metric entropy for every
 invariant probability measure;
\item no Smale horseshoe, full-shift subsystem, or positive-entropy
 symbolic factor;
\item no nontrivial transitive or positive-entropy strange attractor;
\item only equilibrium-supported ergodic invariant probability measures.
\end{enumerate}
Consequently, the system admits no bounded chaotic dynamics under any
of the standard topological, metric, symbolic, recurrence-based, or
attractor-based definitions listed above.
\end{theorem}

\begin{proof}
Symmetry of the Jacobian implies $F=\nabla V$ for a polynomial potential
of degree at most three.  Every bounded trajectory converges to one
equilibrium by
Corollary~\ref{cor:cubic-bounded-convergence}.  Item~\textup{(i)} is
Proposition~\ref{prop:immediate-exclusions}; item~\textup{(ii)} follows
from Propositions~\ref{prop:immediate-exclusions} and
\ref{prop:nonwandering-critical}; item~\textup{(iii)} follows from
Theorem~\ref{thm:no-transitivity} and
Corollary~\ref{cor:no-mixing}; item~\textup{(iv)} is
Theorem~\ref{thm:no-devaney-ay}; item~\textup{(v)} follows from
Corollary~\ref{cor:no-li-yorke-mean} and
Theorem~\ref{thm:no-distributional-chaos}; item~\textup{(vi)} follows
from Theorems~\ref{thm:zero-entropy} and
\ref{thm:invariant-measures-equilibria}; item~\textup{(vii)} is
Corollary~\ref{cor:no-horseshoe}; item~\textup{(viii)} follows from
Proposition~\ref{prop:compact-attractors}; and item~\textup{(ix)} is contained in
Theorem~\ref{thm:invariant-measures-equilibria}.
\end{proof}

\begin{corollary}[Precise Sprott--Zeraoulia statement]
\label{cor:sprott-zeraoulia-resolved}
Three-dimensional quadratic continuous-time systems with symmetric
Jacobian matrices cannot exhibit bounded chaos in any of the standard
rigorous senses covered by
Theorem~\ref{thm:exclusion-standard-chaos}.  In particular, they cannot
possess a compact chaotic attractor.
\end{corollary}



\section{Dynamics at Infinity and the Poincar\'e Compactification}
\label{sec:infinity}

The preceding sections describe completely every trajectory that has a
finite accumulation point.  By
Corollary~\ref{cor:finite-or-infinity-dichotomy}, the only remaining
possibility is genuine escape to infinity.  We now compactify the phase
space and prove that the gradient structure survives on the sphere at
infinity.

The Poincar\'e compactification identifies $\R^3$ with the upper
hemisphere
\[
 \mathbb S^3_+
 =\{(y,s)\in\mathbb S^3:s>0\}
\]
through the central projection
\begin{equation}
 h(\bx)=\frac{(\bx,1)}{\sqrt{1+\|\bx\|^2}}.
 \label{eq:poincare-central-projection}
\end{equation}
The equator
\begin{equation}
 \mathbb S^2_\infty
 =\{(u,0)\in\mathbb S^3:u\in\mathbb S^2\}
 \label{eq:sphere-at-infinity}
\end{equation}
represents the set of directions at infinity.  Equivalently, one may
work in the closed Poincar\'e ball $\overline{\mathbb B^3}$, whose
boundary is $\mathbb S^2_\infty$.  After the standard positive
reparametrization of time, a polynomial vector field extends to a smooth
vector field on this compact manifold and the equator is invariant
\cite{GarciaPerezChavelaSusin2006,BravoFernandezTeruel2020,Li2019}.

\subsection{Homogeneous decomposition and desingularized equations}

Let $F=\nabla V$ have exact degree $d\geq1$.  In the present conjecture
$d\leq2$.  Write
\begin{equation}
 F=F_0+F_1+\cdots+F_d,
 \qquad
 V=V_0+V_1+\cdots+V_{d+1},
 \label{eq:homogeneous-decomposition-infinity}
\end{equation}
where $F_j$ and $V_j$ are homogeneous of degrees $j$ and $j$,
respectively.  Since $F=\nabla V$,
\begin{equation}
 F_j=\nabla V_{j+1},
 \qquad j=0,\ldots,d.
 \label{eq:homogeneous-gradient-components}
\end{equation}

Near infinity, introduce polar variables
\begin{equation}
 \bx=r u,
 \qquad r>0,
 \qquad u\in\mathbb S^2,
 \qquad \rho=\frac1r.
 \label{eq:radial-infinity-coordinates}
\end{equation}
Let
\[
 P_u=I-uu^{T}
\]
be the orthogonal projection onto $T_u\mathbb S^2$.  The original time
$t$ is replaced by the desingularized time $\tau$ defined by
\begin{equation}
 \frac{\dd\tau}{\dd t}=r^{d-1}.
 \label{eq:poincare-time-rescaling}
\end{equation}
For an exact quadratic field, this is $\dd\tau/\dd t=r$.

\begin{proposition}\label{prop:radial-compactified-equations}
In the variables $(u,\rho)$ and the time $\tau$, the compactified system
is
\begin{align}
 u'
 &=P_uF_d(u)+\rho P_uF_{d-1}(u)+\cdots+\rho^dP_uF_0,
 \label{eq:compactified-angular}\\
 \rho'
 &=-\rho\langle u,F_d(u)\rangle
   -\rho^2\langle u,F_{d-1}(u)\rangle
   -\cdots
   -\rho^{d+1}\langle u,F_0\rangle,
 \label{eq:compactified-radial}
\end{align}
where the prime denotes differentiation with respect to $\tau$.
Consequently, $\rho=0$ is invariant and the flow on
$\mathbb S^2_\infty$ depends only on the highest homogeneous component
$F_d$.
\end{proposition}

\begin{proof}
From $\bx=ru$ and $\|u\|=1$, orthogonal projection onto the radial and
tangential directions gives
\[
 \dot r=\langle u,F(ru)\rangle,
 \qquad
 \dot u=\frac1rP_uF(ru).
\]
Using $F(ru)=\sum_{j=0}^dr^jF_j(u)$ and
$\dd t/\dd\tau=r^{1-d}$ yields
\[
 u'=\sum_{j=0}^dr^{j-d}P_uF_j(u),
\]
which is \eqref{eq:compactified-angular} after substituting
$\rho=1/r$.  Moreover,
\[
 \dot\rho=-\frac{\dot r}{r^2},
\]
and the same substitution gives
\eqref{eq:compactified-radial}.  Every term in the latter equation has
a factor $\rho$, proving invariance of the equator.
\end{proof}

\subsection{The spherical gradient structure at infinity}

Let
\begin{equation}
 W(u)=V_{d+1}(u),
 \qquad u\in\mathbb S^2.
 \label{eq:spherical-potential}
\end{equation}
Thus $W=V_{d+1}|_{\mathbb S^2}$ is the restriction of the highest
homogeneous part of the Euclidean potential.

\begin{theorem}\label{thm:spherical-gradient-infinity}
The restriction of the Poincar\'e compactification of $F=\nabla V$ to
$\mathbb S^2_\infty$ is the spherical gradient system
\begin{equation}
 u'=\gradS W(u).
 \label{eq:spherical-gradient-flow}
\end{equation}
In particular,
\begin{equation}
 \frac{\dd}{\dd\tau}W(u(\tau))
 =\|\gradS W(u(\tau))\|^2\geq0,
 \label{eq:spherical-monotonicity}
\end{equation}
with equality precisely at equilibria of the boundary flow.
\end{theorem}

\begin{proof}
By Proposition~\ref{prop:radial-compactified-equations}, the boundary
field is
\[
 u'=P_uF_d(u).
\]
Using \eqref{eq:homogeneous-gradient-components},
$F_d=\nabla V_{d+1}$.  The Riemannian gradient on the unit sphere is the
tangential projection of the Euclidean gradient; hence
\[
 \gradS W(u)
 =P_u\nabla V_{d+1}(u)
 =P_uF_d(u).
\]
This proves \eqref{eq:spherical-gradient-flow}.  Taking the derivative
of $W$ along this flow gives \eqref{eq:spherical-monotonicity}.
\end{proof}

\subsection{Equilibria and critical directions at infinity}

\begin{proposition}\label{prop:equilibria-infinity}
A direction $u\in\mathbb S^2$ is an equilibrium at infinity if and only
if
\begin{equation}
 \nabla V_{d+1}(u)=(d+1)V_{d+1}(u)u.
 \label{eq:equilibrium-infinity-general}
\end{equation}
For an exact quadratic vector field,
\begin{equation}
 \nabla V_3(u)=3V_3(u)u.
 \label{eq:equilibrium-infinity-quadratic}
\end{equation}
\end{proposition}

\begin{proof}
The equilibrium condition is
$P_u\nabla V_{d+1}(u)=0$, so the Euclidean gradient is parallel to $u$:
\[
 \nabla V_{d+1}(u)=\lambda u.
\]
Euler's identity for a homogeneous polynomial of degree $d+1$ gives
\[
 \lambda
 =\langle u,\nabla V_{d+1}(u)\rangle
 =(d+1)V_{d+1}(u),
\]
which proves the result.
\end{proof}

\subsection{Dynamical consequences on the sphere at infinity}

The boundary potential $W$ is real analytic and the sphere is compact.
The \L{}ojasiewicz argument of Section~\ref{sec:dynamical-gradient}
therefore applies in local analytic charts on $\mathbb S^2$.

\begin{theorem}\label{thm:convergence-sphere-infinity}
Every trajectory of the boundary flow
\eqref{eq:spherical-gradient-flow} has finite forward length and
converges to a single equilibrium at infinity.  In particular, the
boundary flow has no nonconstant periodic or homoclinic orbit and no
heteroclinic cycle.
\end{theorem}

\begin{proof}
Every trajectory is bounded because $\mathbb S^2$ is compact.  The
identity \eqref{eq:spherical-monotonicity} is the Riemannian analogue of
\eqref{eq:monotonicity}.  The local \L{}ojasiewicz gradient inequality
for the analytic function $W$ in analytic coordinate charts gives the
finite-length estimate exactly as in
Theorem~\ref{thm:analytic-gradient-convergence}.  Thus every trajectory
converges to one critical point.  The exclusions follow from strict
monotonicity exactly as in
Proposition~\ref{prop:immediate-exclusions}.
\end{proof}

\begin{theorem}\label{thm:no-chaos-sphere-infinity}
On every compact invariant subset of $\mathbb S^2_\infty$, the boundary
flow has:
\begin{enumerate}[label=\textup{(\roman*)}]
\item no nonstationary recurrent, nonwandering, or chain recurrent
 dynamics;
\item no topological transitivity, weak mixing, or topological mixing on
 a set containing more than one point;
\item no Devaney, Auslander--Yorke, Li--Yorke, mean Li--Yorke, or
 distributional chaos;
\item zero metric and topological entropy for every nonzero time map;
\item no Smale horseshoe, positive-entropy symbolic factor, or
 nontrivial transitive strange attractor;
\item only equilibrium-supported ergodic invariant probability measures.
\end{enumerate}
\end{theorem}

\begin{proof}
The boundary system is an analytic gradient flow on a compact Riemannian
manifold.  The proofs of Sections~\ref{sec:dynamical-gradient} and
\ref{sec:exclusion-chaos} use only strict gradient monotonicity,
compactness, analytic convergence, recurrence, and the variational
principle.  They therefore apply verbatim with $V$ replaced by $W$ and
$\R^3$ replaced by $\mathbb S^2$.
\end{proof}

\subsection{Escaping trajectories and their compactified limit sets}

Let $\widehat\phi_\tau$ denote the desingularized compactified flow.  If
an original trajectory satisfies
$\|\bx(t)\|\to+\infty$, its compactified image approaches the equator.
For exact quadratic fields, the desingularized time tends to infinity
even when the original solution blows up in finite time.

\begin{lemma}\label{lem:desingularized-time-escape}
Let $F$ have degree at most two and let
$\|\bx(t)\|\to+\infty$ as $t\uparrow T_{\max}$.  For an exact quadratic
field, choose $t_0$ so large that $\|\bx(t)\|>1$ for $t\geq t_0$ and define
\[
 \tau(t)=\tau(t_0)+\int_{t_0}^t\|\bx(s)\|\dd s.
\]  Then
$\tau(t)\to+\infty$ as $t\uparrow T_{\max}$.
\end{lemma}

\begin{proof}
If $T_{\max}=+\infty$, the conclusion is immediate because
$\|\bx(t)\|\geq1$ for all sufficiently large $t$.  Suppose that
$T_{\max}<+\infty$.  Since $F$ is quadratic, there is $C>0$ such that
\[
 \|F(\bx)\|\leq C(1+\|\bx\|^2).
\]
For $r(t)=\|\bx(t)\|$ sufficiently large,
$\dot r(t)\leq2Cr(t)^2$ wherever $r$ is differentiable.  Hence
$y(t)=1/r(t)$ satisfies $y'(t)\geq-2C$.  Since $y(t)\to0$ as
$t\uparrow T_{\max}$,
\[
 y(t)\leq2C(T_{\max}-t),
 \qquad
 r(t)\geq\frac{1}{2C(T_{\max}-t)}.
\]
Therefore $\int^t r(s)\dd s$ diverges logarithmically as
$t\uparrow T_{\max}$.
\end{proof}

\begin{theorem}\label{thm:escape-limit-critical-infinity}
Let $F=\nabla V$ be a polynomial gradient field of exact degree
$d\in\{1,2\}$, and let $\bx(t)$ be a maximal forward trajectory that
escapes to infinity.  Let
$\widehat\omega(\bx_0)$ be the omega-limit set of its compactified orbit
in desingularized time.  Then
\begin{equation}
 \widehat\omega(\bx_0)
 \subset
 \operatorname{Crit}_{\mathbb S^2}(W)
 \subset\mathbb S^2_\infty,
 \qquad
 W=V_{d+1}|_{\mathbb S^2}.
 \label{eq:escape-omega-critical-infinity}
\end{equation}
The set $\widehat\omega(\bx_0)$ is nonempty, compact, connected, and
invariant.  If the critical points of $W$ are isolated, then there is a
unique equilibrium $u_\infty\in\mathbb S^2$ such that
\begin{equation}
 \frac{\bx(t)}{\|\bx(t)\|}\longrightarrow u_\infty
 \qquad\text{as }t\uparrow T_{\max}.
 \label{eq:unique-escape-direction}
\end{equation}
\end{theorem}

\begin{proof}
The compactified phase space is compact and the desingularized flow is
complete.  Since $\|\bx(t)\|\to\infty$, every accumulation point of the
compactified orbit belongs to the invariant equator.  Standard
properties of omega-limit sets give nonemptiness, compactness,
connectedness, and invariance.

An omega-limit set of a precompact orbit is internally chain transitive
\cite{Conley1978}.
Because the set lies in the invariant equator, this chain transitivity is
with respect to the boundary flow.  By Conley theory, the chain recurrent
set of a compact gradient flow is contained in the critical set of its
strict Lyapunov function.  The strict Lyapunov function here is $W$ by
\eqref{eq:spherical-monotonicity}.  This proves
\eqref{eq:escape-omega-critical-infinity}.

If the critical set is discrete, the connected set
$\widehat\omega(\bx_0)$ is a singleton $\{u_\infty\}$.  The compactified
orbit therefore converges to $u_\infty$, which is equivalent to
\eqref{eq:unique-escape-direction}.
\end{proof}

Without isolation, Theorem~\ref{thm:escape-limit-critical-infinity}
places the compactified omega-limit set inside a connected subset of
spherical equilibria, but does not assert convergence to one direction.
This distinction is necessary because the lower-degree terms of the
original field perturb the angular equation away from the equator.  The
flow \emph{on} the equator always converges by
Theorem~\ref{thm:convergence-sphere-infinity}; an interior escaping orbit
is only asymptotic to that boundary system.

\subsection{Invariant measures and entropy of the full compactification}

Although the complete compactified vector field need not itself be the
gradient of one globally defined potential, its recurrent measure-
theoretic core is still stationary.

\begin{theorem}
\label{thm:compactified-invariant-measures}
Let $\widehat\phi$ be the Poincar\'e compactification of a nonconstant
polynomial gradient field on $\R^3$, and put
\begin{equation}
 \widehat{\mathcal E}
 =h\bigl(\operatorname{Crit}(V)\bigr)
 \cup
 \operatorname{Crit}_{\mathbb S^2}(W).
 \label{eq:compactified-equilibrium-set}
\end{equation}
Every invariant Borel probability measure $\mu$ of a nonzero
compactified time map is concentrated on $\widehat{\mathcal E}$:
\[
 \mu(\widehat{\mathcal E})=1.
\]
Every ergodic invariant probability measure is a Dirac mass at a finite
or infinite equilibrium.  Consequently, every compactified time map has
zero metric and topological entropy and possesses no Smale horseshoe or
positive-entropy symbolic subsystem.
\end{theorem}

\begin{proof}
Let $\mu$ be invariant under a nonzero time map.  Poincar\'e recurrence
implies that $\mu$-almost every point is recurrent.  In the interior of
the compactification, the desingularized vector field is a positive
scalar multiple of the original field.  Hence, along every
nonstationary interior orbit,
\[
 \frac{\dd}{\dd\tau}V
 =a(\bx)\|\nabla V(\bx)\|^2>0
\]
for a positive function $a$.  The recurrence argument of
Proposition~\ref{prop:immediate-exclusions} therefore shows that every
recurrent interior point is a finite equilibrium.  On the invariant
equator, Theorem~\ref{thm:spherical-gradient-infinity} shows that every
recurrent point is a critical point of $W$.  Thus
$\mu(\widehat{\mathcal E})=1$.

The compactified time map is the identity on
$\widehat{\mathcal E}$.  Hence every ergodic invariant measure
concentrated there is a Dirac mass and every invariant measure has zero
metric entropy.  The variational principle gives zero topological
entropy.  A horseshoe or a positive-entropy symbolic subsystem would
contradict this conclusion.
\end{proof}

\subsection{Global compactified conclusion}

\begin{theorem}[Finite and infinite structure of the
Sprott--Zeraoulia class]
\label{thm:global-compactified-sprott-zeraoulia}
Let $F\colon\R^3\to\R^3$ be a polynomial vector field of degree at most
two with symmetric Jacobian.  Constant fields satisfy the conclusions
trivially.  For every nonconstant field:
\begin{enumerate}[label=\textup{(\roman*)}]
\item every maximal forward trajectory either converges to a finite
equilibrium or escapes to infinity;
\item the compactified flow on $\mathbb S^2_\infty$ is the analytic
gradient flow generated by the highest homogeneous part of the
potential;
\item every boundary trajectory converges to one equilibrium at
infinity and the sphere at infinity supports none of the standard
notions of chaos listed in Theorem~\ref{thm:no-chaos-sphere-infinity};
\item every escaping trajectory has its compactified omega-limit set in
the equilibrium set at infinity, and has a unique asymptotic direction
when those equilibria are isolated;
\item every invariant probability measure of the full compactified flow
is concentrated on finite or infinite equilibria, and every compactified
time map has zero topological entropy.
\end{enumerate}
Thus neither bounded dynamics in $\R^3$ nor dynamics intrinsic to the
sphere at infinity can sustain recurrent chaotic behavior, and the full
compactification cannot sustain positive-entropy dynamics.
\end{theorem}

\begin{proof}
A constant field has straight-line trajectories, two antipodal limiting
directions, and no recurrent dynamics.  Assume that $F$ is nonconstant.
Item~\textup{(i)} is
Corollary~\ref{cor:finite-or-infinity-dichotomy}; items~\textup{(ii)} and
\textup{(iii)} are Theorems~\ref{thm:spherical-gradient-infinity},
\ref{thm:convergence-sphere-infinity}, and
\ref{thm:no-chaos-sphere-infinity}; item~\textup{(iv)} is
Theorem~\ref{thm:escape-limit-critical-infinity}; and item~\textup{(v)} is
Theorem~\ref{thm:compactified-invariant-measures}.
\end{proof}



\section{Cubic Potentials and Invariant Algebraic Surfaces}
\label{sec:gradient-darboux}

The preceding sections solve the Sprott--Zeraoulia problem through the
global gradient structure forced by symmetry of the Jacobian, first for
bounded dynamics and then, in Section~\ref{sec:infinity}, for the
dynamics at infinity.  We now relate this approach to the
Darboux-theoretic method developed in
\cite{MessiasSilva2018,MessiasSilva2020,MessiasSilva2022,SilvaMessias2026}.
The leading homogeneous parts of Darboux polynomials determine invariant
algebraic traces on the sphere at infinity, so the same algebraic
skeleton organizes both finite and escaping dynamics.  The purpose of
this section is to make that relation precise and to show how the global
gradient mechanism complements the earlier algebraic criteria.

The main observation is simple and useful.  If
\[
 X=\nabla V
\]
and an algebraic surface $f=0$ is invariant with cofactor $K$, then the
Darboux equation becomes
\[
 \langle\nabla V,\nabla f\rangle=Kf.
\]
Thus invariant algebraic surfaces are governed directly by the cubic
potential.  The potential supplies a global ordering of trajectories,
whereas the algebraic surfaces identify distinguished geometric subsets
of phase space, such as separatrices, invariant manifolds, and possible
boundaries between basins.

\subsection{The Darboux equation in gradient form}

Let $X$ be a polynomial vector field on $\R^3$.  Its action on a smooth
function $f$ will be denoted by
\[
 X(f)=\langle\nabla f,X\rangle.
\]
We recall the standard definition used in the Darboux theory of
integrability; see, for example, \cite{Llibre2004}.

\begin{definition}[Invariant algebraic surface and cofactor]
\label{def:invariant-algebraic-surface}
A nonconstant polynomial $f\in\R[x,y,z]$ is a \emph{Darboux polynomial}
for $X$ if there exists a polynomial $K\in\R[x,y,z]$ such that
\begin{equation}
 X(f)=Kf.
 \label{eq:darboux-general}
\end{equation}
The algebraic set
\[
 \mathcal S_f=\{\bx\in\R^3:f(\bx)=0\}
\]
is then an invariant algebraic surface, and $K$ is called its
\emph{cofactor}.
\end{definition}

For a gradient field $X=\nabla V$, equation
\eqref{eq:darboux-general} takes the form
\begin{equation}
 \boxed{\;
 \langle\nabla V,\nabla f\rangle=Kf.
 \;}
 \label{eq:gradient-darboux}
\end{equation}
When $K$ is constant, the Darboux polynomial $f$ is an eigenfunction of
the Lie derivative $\mathcal L_{\nabla V}$ acting on polynomials.

\begin{proposition}
\label{prop:intrinsic-gradient-surface}
Let $X=\nabla V$ be a smooth gradient vector field and let
$\mathcal S_f=\{f=0\}$ be an invariant algebraic surface.  Denote by
$\mathcal S_f^{\mathrm{reg}}$ its regular part,
\[
 \mathcal S_f^{\mathrm{reg}}
 =\{\bx:f(\bx)=0,\ \nabla f(\bx)\neq0\}.
\]
Then:
\begin{enumerate}[label=\textup{(\roman*)}]
\item $\nabla V(\bx)$ is tangent to $\mathcal S_f$ at every
$\bx\in\mathcal S_f^{\mathrm{reg}}$;
\item the restriction of the ambient vector field to the surface is the
intrinsic gradient of the restricted potential,
\begin{equation}
 X|_{\mathcal S_f^{\mathrm{reg}}}
 =\operatorname{grad}_{\mathcal S_f}(V|_{\mathcal S_f^{\mathrm{reg}}});
 \label{eq:intrinsic-gradient}
\end{equation}
\item along every trajectory contained in the regular part of the
surface,
\begin{equation}
 \frac{\dd}{\dd t}V(\bx(t))
 =\left\|\operatorname{grad}_{\mathcal S_f}
        (V|_{\mathcal S_f})(\bx(t))\right\|^2.
 \label{eq:surface-monotonicity}
\end{equation}
\end{enumerate}
\end{proposition}

\begin{proof}
On $\mathcal S_f$, equation \eqref{eq:gradient-darboux} gives
\[
 \langle\nabla f,\nabla V\rangle=0.
\]
At a regular point, $\nabla f$ is normal to the surface.  Hence
$\nabla V$ is tangent to it, proving~\textup{(i)}.  The intrinsic
gradient of $V|_{\mathcal S_f}$ is the orthogonal projection of the
ambient gradient onto the tangent space.  Since the ambient gradient is
already tangent, this projection is $\nabla V$ itself, proving
\textup{(ii)}.  Finally,
\[
 \frac{\dd}{\dd t}V(\bx(t))
 =\langle\nabla V,\dot\bx\rangle
 =\|\nabla V\|^2,
\]
and \eqref{eq:intrinsic-gradient} yields~\textup{(iii)}.
\end{proof}

Equation \eqref{eq:gradient-darboux} also imposes a compatibility
relation between the Hessian of the potential and the geometry of the
invariant surface.

\begin{proposition}
\label{prop:hessian-cofactor-compatibility}
Let $f$ be a Darboux polynomial with cofactor $K$ for the gradient field
$X=\nabla V$.  Then
\begin{equation}
 \Hess V\,\nabla f+\Hess f\,\nabla V
 =f\nabla K+K\nabla f.
 \label{eq:hessian-cofactor-global}
\end{equation}
Consequently, on the invariant surface $\mathcal S_f$,
\begin{equation}
 \Hess V\,\nabla f+\Hess f\,\nabla V
 =K\nabla f.
 \label{eq:hessian-cofactor-surface}
\end{equation}
If $p\in\mathcal S_f^{\mathrm{reg}}$ is an equilibrium, then
\begin{equation}
 \Hess V(p)\nabla f(p)=K(p)\nabla f(p).
 \label{eq:cofactor-normal-eigenvalue}
\end{equation}
Thus $K(p)$ is the eigenvalue of the linearization in the normal
direction to the invariant surface at $p$.
\end{proposition}

\begin{proof}
Taking the gradient of
$\langle\nabla V,\nabla f\rangle=Kf$ gives
\[
 \Hess V\,\nabla f+\Hess f\,\nabla V
 =f\nabla K+K\nabla f,
\]
which is \eqref{eq:hessian-cofactor-global}.  Restriction to $f=0$
gives \eqref{eq:hessian-cofactor-surface}.  At an equilibrium $p$ we
have $\nabla V(p)=0$, and
\eqref{eq:cofactor-normal-eigenvalue} follows.
\end{proof}

\begin{corollary}\label{cor:tangent-normal-splitting}
Under the hypotheses of
Proposition~\ref{prop:hessian-cofactor-compatibility}, let
$p\in\mathcal S_f^{\mathrm{reg}}$ be an equilibrium.  Then the normal
line $\operatorname{span}\{\nabla f(p)\}$ and the tangent plane
$T_p\mathcal S_f$ are invariant under $DF(p)=\Hess V(p)$.  The normal
eigenvalue is $K(p)$.
\end{corollary}

\begin{proof}
The normal line is invariant by
\eqref{eq:cofactor-normal-eigenvalue}.  Since $\Hess V(p)$ is symmetric,
the orthogonal complement of an invariant eigendirection is invariant.
This orthogonal complement is precisely $T_p\mathcal S_f$.
\end{proof}

\subsection{Constant cofactors and the criterion of 2018}

The algebraic criterion introduced in \cite{MessiasSilva2018} uses an
invariant algebraic surface with a nonzero constant cofactor.  In the
gradient setting, its dynamical meaning can be stated very explicitly.

\begin{proposition}\label{prop:constant-cofactor-evolution}
Suppose that $f$ satisfies
\begin{equation}
 \langle\nabla V,\nabla f\rangle=kf,
 \qquad k\in\R.
 \label{eq:constant-cofactor}
\end{equation}
Then every trajectory for which the flow is defined satisfies
\begin{equation}
 f(\phi_t(\bx))=e^{kt}f(\bx).
 \label{eq:exact-f-evolution}
\end{equation}
Equivalently,
\begin{equation}
 I(\bx,t)=f(\bx)e^{-kt}
 \label{eq:darboux-invariant-constant}
\end{equation}
is a Darboux invariant.

If $k\neq0$, the following conclusions hold.
\begin{enumerate}[label=\textup{(\roman*)}]
\item Every bounded forward trajectory has
\begin{equation}
 \omega(\bx)
 \subset\mathcal S_f.
 \label{eq:omega-in-constant-surface}
\end{equation}
\item If $k>0$, every bounded forward trajectory is entirely contained
in $\mathcal S_f$.
\item If $k<0$, every forward trajectory satisfies
\[
 |f(\phi_t(\bx))|=e^{kt}|f(\bx)|\longrightarrow0
\]
as long as it exists for all $t\geq0$.
\item Every compact invariant set is contained in $\mathcal S_f$.
\end{enumerate}
\end{proposition}

\begin{proof}
Along a trajectory,
\[
 \frac{\dd}{\dd t}f(\phi_t(\bx))
 =k f(\phi_t(\bx)).
\]
Solving this scalar linear equation gives
\eqref{eq:exact-f-evolution} and
\eqref{eq:darboux-invariant-constant}.

Assume first that $k<0$.  Then
$f(\phi_t(\bx))\to0$, so every finite accumulation point of the
trajectory lies in $\mathcal S_f$.  If $k>0$ and the trajectory is
bounded, the polynomial $f$ remains bounded on its closure.  Equation
\eqref{eq:exact-f-evolution} then forces $f(\bx)=0$, and invariance gives
$f(\phi_t(\bx))=0$ for all $t$.

Finally, let $K_0$ be a compact invariant set and $\bx\in K_0$.  The
orbit through $\bx$ is complete and remains in $K_0$.  If $k>0$, forward
boundedness gives $f(\bx)=0$.  If $k<0$, apply the same argument in
negative time.  Hence $K_0\subset\mathcal S_f$.
\end{proof}

\begin{corollary}
\label{cor:gradient-refinement-2018}
Let $X=\nabla V$ be a polynomial gradient vector field with an invariant
algebraic surface having nonzero constant cofactor.  Every bounded
forward trajectory converges to a single equilibrium $p$ satisfying
\begin{equation}
 p\in\operatorname{Crit}(V)\cap\mathcal S_f.
 \label{eq:limit-critical-surface}
\end{equation}
Moreover, every compact invariant set is contained in
$\mathcal S_f$ and its recurrent part consists only of equilibria on
that surface.
\end{corollary}

\begin{proof}
Convergence to one equilibrium follows from
Theorem~\ref{thm:analytic-gradient-convergence}.  Localization on
$\mathcal S_f$ follows from
Proposition~\ref{prop:constant-cofactor-evolution}.  The statement about
compact invariant sets and recurrence follows from the same proposition
and Proposition~\ref{prop:nonwandering-critical}.
\end{proof}

\subsection{Invariant affine planes and exact splitting of the potential}

The case of an invariant plane with constant cofactor, which played a
central role in \cite{MessiasSilva2018}, admits a complete normal form
inside the symmetric-Jacobian class.

\begin{theorem}\label{thm:affine-plane-splitting}
Let $V\colon\R^3\to\R$ be a polynomial and let
\[
 \Pi=\{\bx\in\R^3:f(\bx)=0\},
 \qquad
 f(\bx)=\langle a,\bx\rangle+b,
 \qquad a\neq0,
\]
be an affine plane.  The following statements are equivalent.
\begin{enumerate}[label=\textup{(\roman*)}]
\item The plane $\Pi$ is invariant for $X=\nabla V$ with constant
cofactor $k$, that is,
\begin{equation}
 \langle\nabla f,\nabla V\rangle=kf.
 \label{eq:plane-constant-cofactor}
\end{equation}
\item After an affine orthogonal change of coordinates $(x,y,z)$ to
$(u,v,w)$ in which
\begin{equation}
 w=\frac{f(\bx)}{\|a\|},
 \label{eq:normal-coordinate}
\end{equation}
the potential has the separated form
\begin{equation}
 V(u,v,w)=W(u,v)+\frac{k}{2}w^2,
 \label{eq:potential-plane-splitting}
\end{equation}
where $W$ is a polynomial in two variables.
\item In the same coordinates, the flow is the direct product
\begin{equation}
 \begin{aligned}
  \dot u&=W_u(u,v),\\
  \dot v&=W_v(u,v),\\
  \dot w&=kw,
 \end{aligned}
\label{eq:flow-plane-splitting}
\end{equation}
and therefore
\begin{equation}
 \phi_t(u,v,w)=\bigl(\psi_t(u,v),e^{kt}w\bigr),
 \label{eq:product-flow}
\end{equation}
where $\psi_t$ is the planar gradient flow generated by $W$.
\end{enumerate}
If $V$ has degree at most three, then $W$ has degree at most three.
\end{theorem}

\begin{proof}
Choose a point $\bx_0\in\Pi$, let
$n=a/\|a\|$, and complete $n$ to an orthonormal basis
$\{e_1,e_2,n\}$.  Write
\[
 \bx=\bx_0+ue_1+ve_2+wn.
\]
Then $f(\bx)=\|a\|w$ and
\[
 \langle\nabla f,\nabla V\rangle
 =\|a\|\,V_w.
\]
Consequently, \eqref{eq:plane-constant-cofactor} is equivalent to
\[
 V_w=kw.
\]
Integration with respect to $w$ gives
\[
 V(u,v,w)=W(u,v)+\frac{k}{2}w^2,
\]
proving the equivalence of~\textup{(i)} and~\textup{(ii)}.  Taking the
gradient gives \eqref{eq:flow-plane-splitting}, and solving the normal
scalar equation gives \eqref{eq:product-flow}.  Conversely,
\eqref{eq:flow-plane-splitting} implies $\dot w=kw$, so the plane $w=0$ is invariant, and
therefore $\Pi$ is invariant with cofactor $k$.
\end{proof}

\begin{corollary}\label{cor:global-hessian-splitting}
Under the hypotheses of
Theorem~\ref{thm:affine-plane-splitting},
\begin{equation}
 \Hess V(u,v,w)
 =
 \begin{pmatrix}
  \Hess W(u,v)&0\\
  0&k
 \end{pmatrix}.
 \label{eq:block-hessian-plane}
\end{equation}
Hence the unit normal $n$ to the plane is an eigenvector of
$\Hess V(\bx)$ at every point of $\R^3$, with constant eigenvalue $k$.
At every equilibrium on $\Pi$, the remaining two eigenvalues are those
of $\Hess W$.
\end{corollary}

\begin{corollary}\label{cor:plane-equilibria-stability}
In the coordinates of Theorem~\ref{thm:affine-plane-splitting}:
\begin{enumerate}[label=\textup{(\roman*)}]
\item if $k\neq0$, the equilibria are exactly the points
$(u_*,v_*,0)$ such that $\nabla W(u_*,v_*)=0$;
\item if $k<0$, the invariant plane is exponentially attracting in the
normal direction;
\item if $k>0$, the invariant plane is exponentially repelling in the
normal direction;
\item if $k=0$, every parallel plane $w=\mathrm{constant}$ is invariant,
and each critical point of $W$ generates a line of equilibria parallel
to the normal direction.
\end{enumerate}
\end{corollary}

\subsection{Several invariant surfaces and the criterion of 2026}

The criterion of \cite{SilvaMessias2026} extends the one-surface
construction of \cite{MessiasSilva2018} by allowing several invariant
algebraic surfaces with cofactors satisfying a linear relation.  The
underlying identity can be combined directly with the gradient
convergence theorem.

\begin{proposition}\label{prop:multiple-darboux-surfaces}
Let $f_1,\ldots,f_m$ be Darboux polynomials for a polynomial vector
field $X$, with cofactors $K_1,\ldots,K_m$.  Suppose that there exist
real constants $\lambda_1,\ldots,\lambda_m$, not all zero, and a
constant $s\neq0$ such that
\begin{equation}
 \sum_{j=1}^m\lambda_jK_j=s.
 \label{eq:cofactor-linear-relation}
\end{equation}
On every connected component of
\[
 \R^3\setminus\bigcup_{j=1}^m\mathcal S_{f_j},
\]
the function
\begin{equation}
 \Psi(\bx)=\sum_{j=1}^m\lambda_j\log|f_j(\bx)|
 \label{eq:log-darboux-function}
\end{equation}
satisfies
\begin{equation}
 \frac{\dd}{\dd t}\Psi(\phi_t(\bx))=s.
 \label{eq:log-darboux-linear-time}
\end{equation}
Equivalently,
\begin{equation}
 \exp(\Psi(\bx))e^{-st}
 =\prod_{j=1}^m|f_j(\bx)|^{\lambda_j}e^{-st}
 \label{eq:multiple-darboux-invariant}
\end{equation}
is constant along every orbit that does not lie on one of the surfaces.

If a forward trajectory is bounded, then
\begin{equation}
 \omega(\bx)
 \subset
 \bigcup_{j=1}^m\mathcal S_{f_j}.
 \label{eq:omega-multiple-surfaces}
\end{equation}
An analogous statement holds for bounded backward trajectories and
alpha-limit sets.
\end{proposition}

\begin{proof}
Where all $f_j$ are nonzero,
\[
 X(\log|f_j|)=\frac{X(f_j)}{f_j}=K_j.
\]
Therefore
\[
 X(\Psi)=\sum_{j=1}^m\lambda_jK_j=s,
\]
which proves \eqref{eq:log-darboux-linear-time} and
\eqref{eq:multiple-darboux-invariant}.

Each invariant surface cannot be crossed by an orbit starting outside
it, because along such an orbit
\[
 f_j(\phi_t(\bx))
 =f_j(\bx)\exp\left(\int_0^tK_j(\phi_r(\bx))\dd r\right).
\]
Suppose that a bounded forward trajectory has a point
$q\in\omega(\bx)$ outside the union of the surfaces.  There exists a
sequence $t_n\to\infty$ such that $\phi_{t_n}(\bx)\to q$.  By continuity,
$\Psi(\phi_{t_n}(\bx))\to\Psi(q)$, which is finite.  On the other hand,
\eqref{eq:log-darboux-linear-time} gives
\[
 \Psi(\phi_{t_n}(\bx))=\Psi(\bx)+st_n,
\]
whose absolute value tends to infinity.  This contradiction proves
\eqref{eq:omega-multiple-surfaces}.  The backward statement is
analogous.
\end{proof}

\begin{corollary}\label{cor:gradient-multiple-localization}
Let $X=\nabla V$ be a polynomial gradient field and suppose that the
hypotheses of Proposition~\ref{prop:multiple-darboux-surfaces} hold.
Then every bounded forward trajectory converges to one equilibrium $p$
and
\begin{equation}
 p\in
 \operatorname{Crit}(V)
 \cap
 \bigcup_{j=1}^m\mathcal S_{f_j}.
 \label{eq:critical-point-multiple-surfaces}
\end{equation}
\end{corollary}

\begin{proof}
The gradient convergence theorem gives
$\omega(\bx)=\{p\}$ with $p\in\operatorname{Crit}(V)$.  Proposition
\ref{prop:multiple-darboux-surfaces} places the same omega-limit set in
the union of the invariant surfaces.
\end{proof}

\subsection{Darboux traces on the sphere at infinity}
\label{subsec:darboux-infinity}

The Darboux structure has a natural projective trace at infinity.  Since
the compactified boundary dynamics was derived in
Section~\ref{sec:infinity}, we can now identify how the highest
homogeneous part of a Darboux polynomial restricts to an invariant
algebraic trace on the sphere at infinity.

Let $X=\nabla V$ have exact degree $d\geq1$, and write
\[
 V=V_0+V_1+\cdots+V_{d+1},
 \qquad
 F_d=\nabla V_{d+1},
\]
where the subscript denotes homogeneous degree.  Let $f$ be a Darboux
polynomial of degree $m$ with cofactor $K$.  Since
$\deg X(f)\leq m+d-1$, one has $\deg K\leq d-1$.  Denote the highest
homogeneous parts by $f_m$ and $K_{d-1}$, allowing
$K_{d-1}\equiv0$ when $\deg K<d-1$.

\begin{proposition}\label{prop:darboux-trace-infinity}
Put
\[
 W=V_{d+1}|_{\mathbb S^2},
 \qquad
 g=f_m|_{\mathbb S^2}.
\]
Then, along the spherical gradient field generated by $W$,
\begin{equation}
 \left\langle\gradS W,\gradS g\right\rangle
 =\left(K_{d-1}|_{\mathbb S^2}-m(d+1)W\right)g.
 \label{eq:spherical-darboux-equation}
\end{equation}
Consequently, the algebraic trace
\begin{equation}
 \mathcal S_{f,\infty}
 =\{u\in\mathbb S^2:f_m(u)=0\}
 \label{eq:darboux-trace-set}
\end{equation}
is invariant under the compactified flow at infinity.
\end{proposition}

\begin{proof}
Taking the homogeneous part of degree $m+d-1$ in the Darboux equation
$\langle\nabla V,\nabla f\rangle=Kf$ gives
\begin{equation}
 \left\langle\nabla V_{d+1},\nabla f_m\right\rangle
 =K_{d-1}f_m.
 \label{eq:leading-darboux-equation}
\end{equation}
For $u\in\mathbb S^2$, homogeneity and the orthogonal decomposition into
radial and tangential parts yield
\[
 \nabla V_{d+1}(u)=\gradS W(u)+(d+1)W(u)u
\]
and
\[
 \nabla f_m(u)=\gradS g(u)+m g(u)u.
\]
The tangential and radial components are orthogonal.  Substitution in
\eqref{eq:leading-darboux-equation} gives
\eqref{eq:spherical-darboux-equation}.  Since the derivative of $g$
along the boundary flow is the left-hand side of
\eqref{eq:spherical-darboux-equation}, the zero set of $g$ is invariant.
\end{proof}

\begin{corollary}
\label{cor:quadratic-darboux-trace}
For an exact quadratic gradient field, $d=2$ and
\begin{equation}
 \left\langle\gradS(V_3|_{\mathbb S^2}),\gradS(f_m|_{\mathbb S^2})\right\rangle
 =\left(K_1-3mV_3\right)f_m
 \quad\text{on }\mathbb S^2.
 \label{eq:quadratic-spherical-darboux}
\end{equation}
If the original cofactor is constant, then $K_1\equiv0$, and the
spherical cofactor of the trace is $-3mV_3|_{\mathbb S^2}$.
\end{corollary}

\subsection{The jerk classes of 2020}

The applications to polynomial jerk equations in
\cite{MessiasSilva2020} illustrate that the Darboux criterion is much
broader than the symmetric-Jacobian condition.  In the natural
companion coordinates, a jerk equation is written as
\begin{equation}
 \dot x=y,
 \qquad
 \dot y=z,
 \qquad
 \dot z=j(x,y,z),
 \label{eq:jerk-companion}
\end{equation}
and has Jacobian
\begin{equation}
 DF(x,y,z)=
 \begin{pmatrix}
  0&1&0\\
  0&0&1\\
  j_x&j_y&j_z
 \end{pmatrix}.
 \label{eq:jerk-jacobian}
\end{equation}
This matrix cannot be symmetric, since its $(1,2)$ entry equals $1$
whereas its $(2,1)$ entry equals $0$.

Thus the natural jerk representation lies outside the
Sprott--Zeraoulia gradient class, although it may possess invariant
algebraic surfaces and Darboux invariants.  The results of
\cite{MessiasSilva2020} therefore emphasize the generality of the
algebraic method: invariant surfaces may enforce nonchaotic behavior
even when no global potential exists.  Conversely, in the
symmetric-Jacobian class, the potential excludes bounded chaos even when
no useful Darboux polynomial has yet been identified.

\subsection{Invariant surfaces, symmetry breaking, and the transition
 to chaos}

The Lorenz-like systems studied in \cite{MessiasSilva2022} possess
invariant algebraic surfaces containing stable or unstable invariant
manifolds of equilibria.  The destruction of such surfaces under
parameter perturbation can deform those manifolds and reorganize the
global phase portrait, opening a route from nonchaotic to chaotic
behavior.

The gradient formulation clarifies what this mechanism can and cannot
do inside the Sprott--Zeraoulia class.  Destroying one particular
invariant algebraic surface does not remove the strict potential if the
Jacobian remains symmetric.  Therefore it cannot, by itself, produce
bounded chaos. Indeed, the following result holds. 

\begin{proposition}\label{prop:symmetry-breaking-chaos}
Let $F_\varepsilon$ be a family of quadratic polynomial vector fields on
$\R^3$.  If, for a parameter value $\varepsilon_*$, the system generated
by $F_{\varepsilon_*}$ possesses bounded chaotic dynamics in any of the
standard senses excluded by
Theorem~\ref{thm:exclusion-standard-chaos}, then
\begin{equation}
 DF_{\varepsilon_*}(\bx)
 \neq DF_{\varepsilon_*}(\bx)^T
 \label{eq:symmetry-must-break}
\end{equation}
for at least one point $\bx\in\R^3$.
Equivalently,
\[
 \nabla\times F_{\varepsilon_*}\not\equiv0.
\]
\end{proposition}

\begin{proof}
If the Jacobian were symmetric everywhere, then
$F_{\varepsilon_*}=\nabla V_{\varepsilon_*}$ for a cubic polynomial
potential, and Theorem~\ref{thm:exclusion-standard-chaos} would exclude
the assumed bounded chaotic dynamics.
\end{proof}

\subsection{A unified gradient--Darboux statement}

The preceding results can be summarized in a form that displays the
complementarity between the present approach and the algebraic criteria
of \cite{MessiasSilva2018,MessiasSilva2020,MessiasSilva2022,SilvaMessias2026}.

\begin{theorem}[Gradient--Darboux phase-space structure]
\label{thm:gradient-darboux-unified}
Let $F$ be a quadratic polynomial vector field on $\R^3$ with symmetric
Jacobian, and write $F=\nabla V$ with $\deg V\leq3$.
\begin{enumerate}[label=\textup{(\roman*)}]
\item Every regular invariant algebraic surface $\mathcal S_f$ carries
the intrinsic gradient flow generated by $V|_{\mathcal S_f}$.
\item At a regular equilibrium $p\in\mathcal S_f$, the cofactor value
$K(p)$ is the normal eigenvalue of $DF(p)=\Hess V(p)$.
\item If $\mathcal S_f$ has a nonzero constant cofactor, every compact
invariant set is contained in $\mathcal S_f$, and every bounded forward
trajectory converges to an equilibrium in
$\operatorname{Crit}(V)\cap\mathcal S_f$.
\item If $\mathcal S_f$ is an affine plane with constant cofactor $k$,
the potential and flow split as
\[
 V(u,v,w)=W(u,v)+\frac{k}{2}w^2,
 \qquad
 \phi_t(u,v,w)=\bigl(\psi_t(u,v),e^{kt}w\bigr).
\]
\item If several invariant algebraic surfaces have cofactors satisfying
\eqref{eq:cofactor-linear-relation}, every bounded forward trajectory
converges to an equilibrium contained in their union.
\item The highest homogeneous part of every Darboux polynomial defines
an invariant algebraic trace on $\mathbb S^2_\infty$ through
\eqref{eq:spherical-darboux-equation}.
\end{enumerate}
Thus the potential determines the global temporal ordering and
convergence of trajectories, while the invariant algebraic surfaces
determine an explicit geometric skeleton on which the asymptotic phase
portrait is organized.
\end{theorem}

\begin{proof}
Items~\textup{(i)}--\textup{(vi)} follow, respectively, from
Proposition~\ref{prop:intrinsic-gradient-surface},
Proposition~\ref{prop:hessian-cofactor-compatibility},
Corollary~\ref{cor:gradient-refinement-2018},
Theorem~\ref{thm:affine-plane-splitting},
Corollary~\ref{cor:gradient-multiple-localization}, and
Proposition~\ref{prop:darboux-trace-infinity}.
\end{proof}



\section{Conclusions and Future Perspectives}
\label{sec:conclusions}

The Sprott--Zeraoulia conjecture was originally formulated as a statement
about a particular algebraic restriction on three-dimensional quadratic
systems.  The main conclusion of the present work is that this restriction
is not merely a condition on the linearization: it places the entire vector
field inside the rigid class of polynomial gradient flows.  Once this
geometric fact is made explicit, the absence of chaotic dynamics follows
from a sequence of complementary structures.

First, the symmetry of the Jacobian provides a global polynomial potential
$V$.  Second, the identity
\[
 \frac{\dd}{\dd t}V(\bx(t))=\|\nabla V(\bx(t))\|^2
\]
orders every nonstationary trajectory.  The \L{}ojasiewicz gradient
inequality then implies that every bounded forward trajectory converges to a
single equilibrium.  Third, the Poincar\'e compactification shows that the
restriction of the compactified field to the sphere at infinity is again an
analytic gradient flow, now generated by the highest homogeneous part of the
potential.  Finally, invariant algebraic surfaces satisfy the potential
equation
\[
 \langle\nabla V,\nabla f\rangle=Kf,
\]
which connects the Darboux cofactor with the Hessian and the normal dynamics
and provides an explicit algebraic skeleton for both finite and infinite
phase-space structures.

Taken together, these results exclude the standard compact notions of chaos
considered in Section~\ref{sec:exclusion-chaos}: there are no nonstationary
recurrent points, nonconstant periodic or homoclinic orbits, transitive
compact invariant sets, Li--Yorke or distributionally scrambled sets,
positive-entropy invariant measures, Smale horseshoes, positive-entropy
symbolic subsystems, or compact strange attractors in the usual recurrent
sense.  Every invariant probability measure of the full compactified flow is
concentrated on finite or infinite equilibria.  Thus, within the precise
notions of chaos examined in this manuscript, the conjecture is verified not
only for bounded dynamics in $\R^3$ but also for the dynamics intrinsic to
the sphere at infinity.

\subsection{What is genuinely quadratic and three-dimensional?}

The original conjecture concerns quadratic fields in $\R^3$, and this setting
is responsible for the explicit $19$-parameter family, the cubic potential,
and the two-dimensional sphere at infinity.  The mechanism that excludes
chaos, however, is substantially more general.

Let $n\geq2$ and let
\[
 F\colon\R^n\longrightarrow\R^n
\]
be a polynomial vector field whose Jacobian is symmetric.  Since $\R^n$ is
simply connected, there exists a polynomial potential $V$ such that
$F=\nabla V$.  If $F$ has degree at most $d$, then $V$ has degree at most
$d+1$.  The vector space of all such fields therefore has dimension
\[
 \binom{n+d+1}{n}-1,
\]
the number of coefficients of a polynomial of degree at most $d+1$ in $n$
variables, modulo its additive constant.

The finite-space arguments of Section~\ref{sec:dynamical-gradient} use only
analyticity, the gradient identity, compactness of the relevant orbit, and
the \L{}ojasiewicz inequality.  They do not use either $n=3$ or $d=2$.
Likewise, if $F$ has exact degree $d\geq1$ and
\[
 V=V_{d+1}+V_d+\cdots+V_0
\]
is its homogeneous decomposition, the calculation of
Section~\ref{sec:infinity} extends to $\R^n$ and gives, on the sphere at
infinity $\mathbb S^{n-1}_\infty$,
\begin{equation}
 u'
 =\bigl(I-uu^T\bigr)\nabla V_{d+1}(u)
 =\operatorname{grad}_{\mathbb S^{n-1}}
   \bigl(V_{d+1}|_{\mathbb S^{n-1}}\bigr)(u).
 \label{eq:general-spherical-gradient}
\end{equation}
Thus the boundary dynamics remains gradient for arbitrary polynomial degree
and arbitrary finite dimension.

\begin{proposition}
\label{prop:dimension-degree-independent-core}
Let $n\geq2$ and let $F\colon\R^n\to\R^n$ be a nonconstant polynomial
vector field with symmetric Jacobian.  Then:
\begin{enumerate}[label=\textup{(\roman*)}]
\item $F=\nabla V$ for a polynomial potential $V$;
\item every bounded forward trajectory converges to a single equilibrium;
\item on every compact invariant subset of $\R^n$, the nonwandering set is
contained in $\operatorname{Crit}(V)$, every ergodic invariant probability
measure is a Dirac mass at an equilibrium, and every time map has zero
topological entropy;
\item the flow induced by the Poincar\'e compactification on
$\mathbb S^{n-1}_\infty$ is the analytic gradient flow in
\eqref{eq:general-spherical-gradient}; consequently, every trajectory of the
boundary flow converges to one of its equilibria and the boundary supports no
nonstationary recurrent or positive-entropy dynamics.
\end{enumerate}
\end{proposition}

\begin{proof}
The global potential follows from the Poincar\'e lemma, and polynomiality and
the degree estimate follow from the radial integral construction used in
Section~\ref{sec:gradient-structure}.  Items~\textup{(ii)} and
\textup{(iii)} follow from the same \L{}ojasiewicz, recurrence, and
variational-principle arguments used in Sections~\ref{sec:dynamical-gradient}
and~\ref{sec:exclusion-chaos}; none of them depends on the dimension.  The
radial--angular decomposition and the standard desingularization give
\eqref{eq:general-spherical-gradient}.  Since
$V_{d+1}|_{\mathbb S^{n-1}}$ is analytic on a compact analytic manifold, the
\L{}ojasiewicz gradient theorem applies to the boundary flow and proves
item~\textup{(iv)}.
\end{proof}

This proposition shows that increasing the degree or the dimension does not
restore chaotic recurrence inside the symmetric-Jacobian class.  What becomes
more difficult is not the fundamental nonchaotic mechanism, but the algebraic
and geometric classification of equilibria, critical manifolds, invariant
hypersurfaces, heteroclinic connections, and escape directions.

\subsection{Polynomial fields of arbitrary degree in \texorpdfstring{$\R^3$}{R3}}

For a polynomial field of degree $d$ in $\R^3$ with symmetric Jacobian, the
potential has degree $d+1$, and the sphere at infinity is governed by
\[
 W=V_{d+1}|_{\mathbb S^2}.
\]
Therefore, the finite and boundary gradient mechanisms established here
persist without an essential change.  Several questions nevertheless become
substantially more involved as $d$ grows.

A first problem is to classify the critical set of $W$ and to determine
algebraic conditions guaranteeing that the equilibria at infinity are
isolated.  Isolation gives a unique asymptotic direction for every escaping
trajectory, whereas nonisolated critical sets require a more delicate study
of how lower-order terms approach the boundary flow.  A second problem is to
obtain quantitative escape and blow-up rates from the homogeneous expansion
of $V$.  A third is to extend the Darboux analysis of
Section~\ref{sec:gradient-darboux} to higher-degree invariant surfaces and to
understand how their leading homogeneous parts stratify the sphere at
infinity.

It would also be useful to determine normal forms for polynomial potentials
under translations, orthogonal changes of coordinates, and, when dynamically
appropriate, more general affine transformations.  In the quadratic-vector-field 
case this becomes a classification of cubic forms together with their
quadratic and linear perturbations.  For higher degree, the corresponding
problem is naturally related to the invariant theory of symmetric tensors and
to the singularity theory of polynomial functions.

\subsection{Higher-dimensional phase spaces}

For $n\geq4$, a general flow on $\mathbb S^{n-1}$ may support complicated
recurrent dynamics, and no Poincar\'e--Bendixson principle is available.
Equation~\eqref{eq:general-spherical-gradient} shows that this difficulty does
not arise from recurrence in the present class: the flow at infinity remains
gradient and hence nonchaotic.  Higher dimension instead permits more
complicated equilibrium geometry, including positive-dimensional critical
sets, Morse--Bott manifolds, higher-dimensional stable and unstable
manifolds, and large heteroclinic complexes ordered by the potential.

A systematic extension to $\R^n$ should therefore focus on the topology and
stratification of
\[
 \operatorname{Crit}(V)
 \quad\text{and}\quad
 \operatorname{Crit}_{\mathbb S^{n-1}}
 \bigl(V_{d+1}|_{\mathbb S^{n-1}}\bigr),
\]
as well as on the Morse and Conley complexes generated by their connecting
orbits.  Another natural direction is the study of invariant algebraic
hypersurfaces
\[
 f=0,
 \qquad
 \langle\nabla V,\nabla f\rangle=Kf,
\]
and the relation between the cofactor, the normal Hessian, and the topology of
the hypersurface in arbitrary dimension.

\subsection{Symmetry breaking and transitions to chaos}

The symmetric-Jacobian family may be viewed as a nonchaotic organizing
skeleton inside the space of polynomial vector fields.  A particularly
important continuation of this work is to study perturbations
\begin{equation}
 \dot\bx=\nabla V(\bx)+\varepsilon G(\bx),
 \label{eq:symmetry-breaking-perturbation}
\end{equation}
for which the antisymmetric part of the Jacobian is no longer zero.  Along a
solution of \eqref{eq:symmetry-breaking-perturbation},
\[
 \frac{\dd}{\dd t}V(\bx(t))
 =\|\nabla V(\bx(t))\|^2
  +\varepsilon\langle\nabla V(\bx(t)),G(\bx(t))\rangle,
\]
so the sign-definite monotonicity responsible for the present results may be
lost.  This formula gives a quantitative starting point for studying the
onset of periodic or recurrent behavior.

Several questions arise naturally: how large must the antisymmetric component
of $DF$ be before recurrence becomes possible?  Which local bifurcations first
appear when the symmetry is broken?  How does the destruction of an invariant
plane or another Darboux surface interact with the loss of the global
potential?  Can one derive perturbative criteria, perhaps using normal forms,
Melnikov functions, or averaging theory, that locate the first periodic,
homoclinic, or chaotic invariant sets?  These questions connect the present
work directly with the transition mechanisms investigated in
\cite{MessiasSilva2022}.

\subsection{Algebraic and computational developments}

The potential formulation also suggests a symbolic and computational program.
Given a polynomial vector field, one may first test the linear conditions
$DF=DF^T$, reconstruct $V$ by integration, compute its highest homogeneous
part, and solve the constrained critical-point equation
\[
 \nabla V_{d+1}(u)
 =(d+1)V_{d+1}(u)u,
 \qquad
 \|u\|=1,
\]
to determine the equilibria at infinity.  One may then search for Darboux
polynomials through the linear equation
\[
 \langle\nabla V,\nabla f\rangle=Kf
\]
and use their traces at infinity to refine the global phase portrait.

For the quadratic three-dimensional family, this procedure could lead to a
parameter-space stratification according to the number and type of finite and
infinite equilibria, the existence of invariant planes or higher-degree
surfaces, and the topology of their connecting orbits.  Such a classification
would complement the universal nonchaotic theorem with an explicit catalogue
of all phase-portrait structures allowed inside the $19$-parameter family.

\subsection{Final perspective}

The earlier Darboux criteria and the present gradient approach play different
but complementary roles.  Invariant algebraic surfaces provide explicit
geometric barriers, separators, and carriers of stable and unstable
manifolds.  The polynomial potential supplies the global ordering mechanism
that excludes recurrent dynamics whether or not such a surface can be found.
The Poincar\'e compactification completes this picture by showing that the
highest homogeneous part of the same potential governs the dynamics at
infinity.

The Sprott--Zeraoulia class is therefore nonchaotic for a structural reason:
it consists of analytic gradient flows in the finite phase space and induces
analytic spherical gradient flows at infinity.  This interpretation gives an analytical resolution of
the original three-dimensional quadratic problem for the compact
topological, metric, symbolic, and measure-theoretic definitions of chaos
considered here,
while opening a broader program on polynomial gradient systems, their
algebraic invariant geometry, and the mechanisms through which symmetry
breaking may create complex dynamics.

\section*{Acknowledgements}

This work was financed, in part, by the S\~ao Paulo Research Foundation (FAPESP), Brazil, Process Numbers 2023/06076-0 and 2024/15612-6.


\bibliographystyle{plainnat}
\bibliography{sprott_zeraoulia_arxiv}

\end{document}